\documentclass[12pt]{amsart}
\usepackage{graphicx}
\usepackage{amsmath}
\usepackage{amstext}
\usepackage{amsfonts}
\usepackage{amsthm}
\usepackage{amssymb}

\newtheorem{theorem}{Theorem}
\newtheorem{definition}{Definition}

\newtheorem{lemma}{Lemma}

\newcommand{\Aff}{\mathcal{A}}
\newcommand{\Ar}{\mathbb{R}}

\newcommand{\Mod}{\mathcal{M}}

\newcommand{\Energy}{\mathcal{E}}

\newcommand{\ep}{\epsilon}

\newcommand{\C}{\mathbb{C}}

\newcommand{\Lag}{\mathcal{L}}
\newcommand{\Flat}{\mathcal{C}}
\newcommand{\delbar}{\overline{\partial}}

\newcommand{\GauS}{\mathcal{G}(\Sigma)}
\newcommand{\Gau}{\mathcal{G}}
\newcommand{\liea}{\mathfrak{g}}
\newcommand{\Moment}{\mu}
\newcommand{\harmo}{\mathcal{H}}

\begin{document}

\title{Gromov-Uhlenbeck Compactness}
\author{Max Lipyanskiy}
\date{}

\address{Simons Center for Geometry and Physics, Stony Brook University, Stony Brook, NY 11794 }
\email{mlipyan@gmail.com}
\maketitle
     
\section{Introduction}
\subsection{The Atiyah-Floer Conjecture}
In the early 1980's Casson (see \cite{Casson}) introduced a new invariant of 3-manifolds based on a "count" of representations of the fundamental group of a 3-manifold to $SU(2)$, or more generally a compact Lie group $G$.  Let us briefly recall the basic idea.   Let $Y$ be a closed oriented 3-manifold and let $$Y=H^+\cup_\Sigma H^-$$ be a Heegaard splitting of $Y$ along a genus $g$ surface $\Sigma$.  Let $\Mod(\Sigma)$ be the character variety of of $\Sigma$.  This is defined as the space of representations $$Hom(\pi_1(\Sigma), SU(2))/SU(2)$$ modulo conjugation by $SU(2)$. As observed by Atiyah and Bott \cite{AB}, $\Mod(\Sigma)$ is a compact K\"{a}hler manifold of dimension $6g-6$ (with singularities).   The surjections $$\pi_1(\Sigma)\rightarrow \pi_1(H^\pm) $$ induce injective maps $$\mathcal{L}^\pm\rightarrow \Mod(\Sigma)$$  where $\mathcal{L}^\pm$ is the set of representations of the free group  $ \pi_1(H^\pm) $.  In fact, each $\mathcal{L}^\pm$ has dimension $3g-3$ and thus has half the dimension of the character variety of $\Sigma$.  In fact, one may show that the spaces $\mathcal{L}^\pm$ are Lagrangian for the natural symplectic form on $\Mod(\Sigma)$.  By considering generic intersections of $\mathcal{L}^\pm$ (in fact, 1/2 of the number of intersections), Casson was able to define a  count of representations, in the case when $Y$ is a integral homology sphere.  \\\\
 Taubes \cite{Taubes} gave a gauge theoretic interpretation of Casson's results.  Let us briefly recall the basic idea.  Let $A$ be a connection on a trivial $SU(2)$-bundle over $Y$ and let $F_A$ be the curvature.  Connections with vanishing curvature are called flat and are classified by their holonomy.  Therefore, we have a natural correspondence betweeen flat connections (module gauge) and representations of $\pi_1(Y)$ into $SU(2)$.  By introducing suitable perturbations Taubes defines a gauge-theoretic count of such flat connections.  In fact, the Casson invariant is an infinite dimensional  analogue of  Euler characteristic of the space of all connections modulo gauge.    
\subsection{Categorifications}
We have arrived at two different geometric descriptions of the Casson invariant - one based on gauge theory and the other on symplectic geometry.  As the reader might anticipate,  both descriptions of the invariant have categorifications that express the invariant as the euler characteristic of a certain homology group.  In both cases, the corresponding homology theory is due to the groundbreaking work of Floer. \\\\  
Let us first discuss the Lagrangian viewpoint.  Given a pair of Lagrangians $\mathcal{L}^\pm$, in a symplectic manifold $\Mod(\Sigma)$, Floer constructs a chain complex $C_*(\mathcal{L}^-, \mathcal{L}^+)$ freely generated by the set $\mathcal{L}^-\cap \mathcal{L}^+$.  Given two intersection points $x, y \in \mathcal{L}^-\cap \mathcal{L}^+$, the differential on  $C_*(\mathcal{L}^-, \mathcal{L}^+)$ counts holomorphic strips $$u:[0,1]\times \Ar \rightarrow \Mod(\Sigma)$$ with $u(0,\cdot ) \in  \mathcal{L}^-$ and $u(1,\cdot ) \in  \mathcal{L}^-$.  Let us denote the resulting groups by $HF_*(\mathcal{L}^-, \mathcal{L}^+)$\\\\
From the gauge theory perspective, we define the chain complex $C_*(Y)$ as follows.  The generators are given by flat connections $A$ on $Y$.  The differential, on the other hand, is the signed count of solutions to the ASD equation $$*F_B=-F_B$$ on the 4-manifold $Y\times \Ar$.  These are required to have finite energy and converge to specified flat connections at the ends. Let us denote the resulting groups by $I_*(Y)$\\\\ Under the assumption of transversality, we note that the generators of the two chain complexes are identical.  However, the corresponding homology theories are based on solutions to nonlinear PDE's in dim 2 and 4 and a priori do not appear to be related.  \\\\
We have the following:\\\\
\textbf{Atiyah-Floer Conjecture \cite{Atiyah}:}  There exists an isomorphism $$I_*(Y) \cong HF_*(\mathcal{L}^-, \mathcal{L}^+)$$ 
The immediate problem with this conjecture is that the relevant group on the symplectic side has not been defined.  This is related to the fact that the presence of reducible representations cause singularities in the character variety (see however \cite{Kat2}). \\\\
 On the other hand, there are several ways of getting around this issue that lead to interesting and well defined groups.  If $b_1(Y)>0$, one approach is to consider nontrivial $U(2)$-bundles over $Y$ with odd $c_1$.  This way, all flat connections are irreducible the relevant spaces have well defined groups.  A proof of the analogue of the conjecture for the case of a mapping torus has been given in \cite{Sal} using adiabatic limits.  This work is part of a series to prove the conjecture for a general $Y$ with a nontrivial $U(2)$-bundle - thus covering all cases where the groups are well-defined.  Our approach is not based on adiabatic techniques but rather develops an analytic setting ((which we call Gromov-Uhlenbeck compactness) that combines the pseudo-holomorphic curves and ASD connections into a unified framework.  As a consequence of the general theory, one constructs an apriori map $$\Phi: I_*(Y)\rightarrow HF_*(\mathcal{L}^-, \mathcal{L}^+) $$ that can be shown to be an isomorphism.   The present work, is devoted to the analytic foundations of such a theory.  Applications to Floer homology will be addressed elsewhere.  We hope, however, that the present techniques are of independent interests and may be of use in other contexts where gauge theory and symplectic geometry interact.  
 
\subsection{Overview of the Results}
We give a brief outline of the main compactness results of this paper.  \\\\
Consider a compact Riemann surface $\Sigma$ with complex structure $j_\Sigma$ and Kahler metric $g_\Sigma$.  Let $E\rightarrow \Sigma$ be a 2-dimensional complex vector bundle with odd $$c_1(E)\in H^2(\Sigma; \mathbb{Z}) \cong \mathbb{Z}$$  Let $V \rightarrow \Sigma$ be the bundle of traceless endomorphisms of $E$.  Note that $V$ has structure group $SO(3)$.   We fix once and for all a unitary connection $\alpha_{det}$ on $det(E)$.  Once this choice of $\alpha_{det}$ is made, we have an identification between $SO(3)$-connections on $V$ and $U(2)$-connections on $E$ that induce $\alpha_{det}$ on $det(E)$.\\\\
Let $\GauS_E$ be the group of $U(2)$-gauge transformations of $E$ that descend to the identity on $det(E)$.  $\GauS$ has a natural induced action on the connections on $V$ and we let $\GauS\subset \GauS_V$ be its image.  We may identify the action of $\GauS$ with the action of $SO(3)$-gauge transformations on $V$ that lift to gauge transformations of $E$.      

\begin{definition}
Let $\Aff$ be the affine space of $SO(3)$-connections on $V$.  
\end{definition}
$\GauS$ acts on $\Aff$ by $$g^*(\alpha)=\alpha+g^{-1}d_{\alpha}g$$
Following Atiyah and Bott \cite{AB}, we note that this action is Hamiltonian with moment map $$\Moment: \Aff \rightarrow \Omega^0(\Sigma; \liea)$$ given by $\Moment(\alpha)=*_2F_\alpha$, where $F_\alpha$ is the curvature of $\alpha$.  
Let $\Flat \subset \Aff$ be the set of projectively flat connections.  
\begin{definition}
Let $\Mod=\Moment^{-1}(0)/\GauS=\Aff // \GauS$
\end{definition}
In fact, $\Mod$ is a compact Kahler manifold (see \cite{AB}).  $\Mod$ has a concrete description in terms of representations of $\pi_1(\Sigma)$.  Pick a point $p\in \Sigma$.  The space $\Mod$ is the space of representations $\pi_1(\Sigma-p)\rightarrow SU(2)$ that have holonomy $-I$ around $p$, modulo conjugation by $SU(2)$.  If we pick a standard homology basis  $\{\alpha_i \}_{i=1}^g,\{\beta_i \}_{i=1}^g$ , we may identify $\Mod$ with the space 
\begin{equation}
\{ g_i\in SU(2), h_i\in SU(2)|\Pi_{i=1}^g[g_i,h_i]=-I \}/SU(2)
\end{equation}

In general, given a symplectic manifold $(M,\omega)$ with a Hamiltonian group action $G$ and corresponding moment map $$\mu:M\rightarrow Lie(G)$$ one may form the symplectic reduction at $0$ by $$M//G=\mu^{-1}(0)/G$$ Provided $G$ acts freely on $\mu^{-1}(0)$, $(M//G,\omega_{M//G})$ inherits a symplectic structure from $M$.  Let $(M//G) ^-$ denote the symplectic manifold with the opposite form. There is a canonical Lagrangian $$\Lag=\mu^{-1}(0)\subset (M//G) ^-\times M$$ defined as the set of pairs $([m],m)$.  Here, $[m]=mG$ denotes the orbit of $m$.  This general construction applies to the case of interest where $M=\Aff$ and $G=\GauS$.  
\begin{definition}
Let $\Lag \subset \Mod\times \Aff$ be the set of pairs $([\alpha],\alpha)$ where $\alpha$ is a flat connection on $\Sigma$.  
\end{definition}   

\subsection{Matching Boundary Conditions}
Let $B_R(p) \subset \C$ be the closed disk of radius $R$ centered at the origin and let $$H^+=\{(s,t)\in \C| s \geq 0  \}$$ be the positive half-plane.  We define $D_R^+$ as $H^+\cap B_R$.  Let $D_R^-$ be the reflection of this disk in the $t$-axis.   In general, $\partial D_R^+=I_R\cup S_R$ where $$I_R=\{(0,t)\in D_R^+ \}$$ and $$S_R=\{(s,t)\in D_R^+| s^2+t^2=R^2\}$$    The \textbf{interior} of a disk $D_R^+$ is the set of points with $s^2+t^2 <R^2$ and will be denoted by $\mathring{D}_R^+$.  \\\\
Consider a holomorphic map $$u: D_R^- \rightarrow \Mod $$ and an ASD connection $A$ on $D_R^+ \times \Sigma$. For each $t\in I_R$, we may restrict $A$ to the slice $(0,t)\times \Sigma$.  This gives us a map $$R_{A}:I_R\rightarrow \Aff(\Sigma)$$
\begin{definition}
The pair $(u,A)$ is said to be matched if at each $t\in I_R$, $u(0,t)=[R_{A}(t)]$ where $[R_{A}(t)]$ denotes the gauge orbit of $R_{A}(t)$.
\end{definition}
Note that this is precisely the condition that $(u(0,t),R_A(t))\in \Lag$ for each $t$.  If $A$ was a holomorphic curve, this would amount to a Lagrangian boundary condition for the pair $$(\tilde{u},A):D_R^+\rightarrow \Aff(\Sigma)$$ where $\tilde{u}(s,t)=u(-s,t)$.   For convenience, we will refer to the pair $(u,A)$ as defined on $D_R^+\times \Sigma$.  Let $$e(s,t)=\frac{1}{2}|du(-s,t)|^2+\frac{1}{2}\int_\Sigma |F_A|^2$$ and $$\Energy(u,A)=\int_{D_R^+}e$$
The following gives a key a priori estimate for matched pairs:
\begin{theorem}

There exists $\hbar,C>0$ with the following property.  Let $(u,A)$ be a matched pair on some $D_R^+ \times \Sigma$.  If $$\Energy(u,A)< \hbar$$ then 
$$e(p)\leq C\Energy(u,A) R^{-2}$$
\end{theorem}

We now state our main compactness results.  We will consider a sequence of matched pairs $(u_i,A_i)$ on $D_R^+ \times \Sigma$.  By the regularity results of section $\ref{secreg}$, we may assume that the sequence consists of smooth elements. 
\begin{definition}
A \textbf{singular set} $S$ on $D_R\times \Sigma$ is a finite collection of points $x_i\in (D_R^--\partial D_R^-)$, $y_i\in (D_R-\partial D_R)\times \Sigma$, $z_i\times \Sigma \in (I_R-\partial I_R)\times \Sigma$.  The $z_i\times \Sigma$ are the \textbf{boundary slices} of $S$.      

\end{definition} 
\begin{theorem}
Assume that we have a uniform bound $E(u_i,A_i) <C$.  There exists a subsequence $(u_j,A_j)$ and a singular set $S$ with the following properties.  Let $K_0$ be a compact set in $\mathring{D}^-_R-S$ and $K_1$ be a compact set in $\mathring{D}_R \times \Sigma -S$.    We have that $u_j$ converges in any $C^k$ norm on $K_0$ and $A_j$ converges in any $C^k$-norm on $K_1$.   Finally, the energy loss at each singular point is at least  $\hbar$ for some sufficiently small $\hbar>0$ independent of the choice of sequence.  
\end{theorem}   

\textbf{Acknowledgement.}  We wish to thank Tom Mrowka, Dusa McDuff and Dennis Sullivan for useful conversations.  In addition, we would like to thank the Simons Center For Geometry and Physics for their hospitality while this work was being completed.

\section{Review of Sobolev Spaces}
\subsection{The Space $L^p_k$}
In this work we will need to consider Sobolev spaces of functions with an infinite dimensional Banach space as target.  Thus, we begin with a brief review of Sobolev spaces with the purpose of setting down notation as well as explaining how the results extend with minimal effort to the infinite dimensional case.  A general reference for Sobolev spaces is \cite{Taylor1}.\\\\
Let $M$ be a closed, oriented, Riemannian manifold of dimension $d$.  For $p>1$ and $k$ a nonnegative integer, one has the Sobolev space  $L^p_k(M)$ of real valued functions on $M$.  These are defined by completing the space of smooth functions on $M$ with respect to the norm:
\begin{equation}
\label{eq7}
\sum_{i=0}^k||\nabla^i f||_{L^p}
\end{equation}
More generally, given a vector bundle $V\rightarrow M$, one may consider the Sobolev space of sections of $V$.  The results of this section apply in this general context. We will often omit $M$ from the notation when the domain is clear in a particular discussion.  Let $C^k(M)$ stand for the Banach space of functions on $M$ with $k$ continuous derivatives. The norm on $f\in C^k$ is given by 
\begin{equation}
\sup_{x\in M, 0\leq  i \leq k}|\nabla^i f(x)|
\end{equation}

  Recall the following (see \cite{Taylor3} for a proof) fundamental theorems:
\begin{theorem}
\label{thm1}
 If $pk<d$, we have the embedding 
 \begin{equation}
 L^p_k\rightarrow L^{\frac{dp}{d-kp}}
 \end{equation}
 If $pk>d$, we have the embedding 
 \begin{equation}
 L^p_k\rightarrow C^0
 \end{equation}
 More generally, if $p(k-m)>d$, we have the embedding 
 \begin{equation}
L^p_k \rightarrow C^m
\end{equation}
\end{theorem}
We have the multiplication map $$L^p\cdot L^{p'}\rightarrow L^{\frac{pp'}{p+p'}}$$ as well as the following theorem:
\begin{theorem}
If $pk>d$, the spaces $L^p_k$ form a Banach algebra under the operation of pointwise multiplication.  
\end{theorem}
At times, it is useful to have a definition of Sobolev spaces for negative $k$:
\begin{definition}
Let $k$ be a negative integer and let $1/p+1/q=1$.  We set $$L^p_k(M)=(L^q_{-k}(M))^*$$ where $(L^q_{-k}(M))^*$ denotes the dual of $L^q_{-k}(M)$.  
\end{definition}

In the case $p=2$, a spectral definition of Sobolev norms is useful.  Let $\Delta=d^*d$ be the scalar Laplacian for functions on $M$  and let $\phi_\lambda$ be an orthonormal eigenbasis of $\Delta$.  For any smooth $f$, we have the decomposition $$f=\sum_\lambda c_\lambda \phi_\lambda$$ with $c_\lambda = \langle \phi_\lambda, f \rangle_{L^2}$.  We may define the $L^2_k$-norm by setting 
\begin{equation}
||f||^2_{L^2_k}=\sum_\lambda |c_\lambda|^2(|\lambda|^2+1)^{k/2}
\end{equation}
Standard elliptic estimates imply that this definition yields a norm equivalent to $(\ref{eq7})$ in the case $p=2$ (see \cite{Taylor1}).   We may therefore alternatively define $L^2_k(M)$ as the completion of smooth functions with respect to this norm.  One advantage of the spectral  definition is that it immediately extends to all $k\in \Ar$. \\\\
Let us now turn to the case of a manifold with boundary.  Let $M^+$ be a compact, oriented, Riemannian manifold with boundary $\partial M^+$.  We will assume that near the boundary, $M^+$ is isometric to $\partial M^+ \times [0,1)$. Let $M^-$ denote a copy of $M^+$ with the opposite orientation.  We let $${M}=M^+\cup_{\partial M^+} M^-$$ be the double manifold formed by gluing two copies $M^+$ and  $M^-$  along the boundary.  \\\\
One can easily extend the definition of Sobolev spaces to the case of a manifold with boundary.  Indeed, for $k \in \{0,1, 2 \dots \}$, one may define $L^p_k(M^+)$ by completing the space of smooth functions on $M^+$ with respect to the norm  $(\ref{eq7})$.  We will make use of the following basic extension lemma:
\begin{lemma}
\label{lem4}
 There exists a continuous linear extension map $$E:L^p_k(M^+)\rightarrow L^p_k(M)$$ such that $E(f)_{|M^+}=f$ for all $f\in L^p_k(M^+)$.  
\end{lemma} 
\begin{proof}
The proof is contained in \cite{Taylor1} and we give a sketch of the construction for later use. For this, we construct $$E:L^p_l(M^+)\rightarrow L^p_l(M)$$  for all $l\leq k$ as follows. First, we locally identify $M^+$ with $(x,y)\in \Ar^{d-1}\times [0,1)$ where $x$ are the $\partial M^+$ coordinates and $y$ is the normal coordinate.  Fix some choice of coefficients $a_j$.  Let $E(f)$ be $f(x,y)$ for $y\geq 0$ and $$\sum_{j=1}^{k+1}
a_j f(x, -jy)$$ for $y<0$.  As explained in \cite{Taylor1}, there exists a unique choice of coefficients $a_j$ so that all derivatives up to order $k$ match up at the boundary for all $f$.
\end{proof}
In view of the previous lemma, we may alternatively define $L^p_{k}(M^+)$ as 
\begin{equation}
L^p_k(M^+)=L^p_k(M)/L^p_k(M)_{M^-}
\end{equation}
where $L^p_k(M)_{M^-}\subset L^p_k(M)$ consists of elements with support in $M^-$.  This definition allows one to extend $L^p_k(M^+)$ to all real $k$.\\\\ Let $$R:C^\infty (M^+)\rightarrow C^\infty (\partial M^+)$$  denote the restriction map.  We have the following trace theorem (see \cite{Taylor1}):
\begin{theorem}
If $k>1/2$, $R$ extends to a continuous surjective map 
\begin{equation} 
R: L^2_k(M^+) \rightarrow L^2_{k-1/2}(\partial M^+)   
\end{equation}
For any $p>1$, there is a continuous restriction map
\begin{equation} 
R: L^p_1(M^+)\rightarrow L^p(\partial M^+) 
\end{equation}
\end{theorem}
\subsection{Some Nonlinear Estimates}
\label{Sobolev}
In dealing with nonlinear estimates, it will be convenient for us to introduce an alternative notation for Sobolev spaces.  To this end, we will often write $L^{1/q}(M)_k$ instead of $L^p(M)_k$ where $pq=1$.  Note the embedding 
\begin{equation}
L^{1/q}\rightarrow L^{1/r}
\end{equation} 
for $q\leq r$.  The multiplication lemma for Sobolev spaces may now be expressed as 
\begin{equation}
L^{1/q_1} \cdot L^{1/q_2}\rightarrow L^{1/(q_1+q_2)}
\end{equation}
as long as $q_1+q_2 \leq 1$, while the basic embedding theorem $\ref{thm1}$ is now expressed as
\begin{equation}
L^{1/q}_1\rightarrow L^{1/(q-d^{-1})}
\end{equation}
for $q-1/d>0$.  \\\\ We turn to some specialized results that we will need in the sequel.  Let $d=4$ and fix $1/4<p_0<1/2$.  We have the embeddings 
\begin{equation}
\label{eq13}
L^{1/p_0}_1\cdot L^{1/p_0}\rightarrow L^{1/(p_0-1/4)} \cdot L^{1/p_0} \rightarrow L^{1/(2p_0-1/4)} 
\end{equation}
\begin{equation}
\label{eqemb1}
L^{1/p_0}_1\cdot L^{1/p_0}_1\rightarrow L^{1/(p_0-1/4)}\cdot L^{1/(p_0-1/4)}\rightarrow L^{1/(2p_0-1/2)} 
\end{equation}
\begin{equation}
L_{2}^{1/(2p_0-1/4)}\rightarrow L_{1}^{1/(2p_0-1/2)}
\end{equation}
In addition, we claim that
\begin{equation}
L^{1/p_0}_1 \cdot L^{1/(2p_0-1/4)}_2\rightarrow L^{1/(2p_0-1/4)}_1
{\label{eqemb2}}
\end{equation}
Indeed, by applying derivatives, we obtain $$L^{1/p_0}\cdot L^{1/(2p_0-1/4)}_2\rightarrow L^{1/(3p_0-3/4)}$$ and $$L^{1/p_0}_1\cdot L^{1/(2p_0-1/4)}_1\rightarrow L^{1/(3p_0-3/4)}$$  Finally, observe that $2p_0-1/4> 3p_0-3/4$ by the fact that $1/2>p_0$. Therefore, we obtain the embedding $$L^{1/(3p_0-3/4)}\rightarrow L^{1/(2p_0-1/4)}$$ and thus $(\ref{eqemb2})$ as desired.
\subsection{Sobolev Spaces for Banach Valued Functions}
\label{SobBan}
Let $B$ be a separable Banach space and let $M$ be a closed Riemannian manifold of dimension $d$. In practice, we will always assume that $B=L^q_l(\Sigma) \times \Ar^N$ for some compact manifold $\Sigma$.   Much of the Sobolev space theory discussed in the previous sections directly generalizes to the case of $$f:M\rightarrow B$$  Below, we explain the bare minimum we will use in this work.  Since many of the proofs are identical to their finite dimensional analogues, we will present a very condensed account.  \\\\
Let $p >1$ and $k$ a nonnegative integer.   On smooth maps $f:M\rightarrow B$, we define the $L^p_k$-norm by $$\sum_{i=0}^k(\int_M |\nabla^i f|^p)^{1/p}$$  Let $L^p_k(M; B)$ denote the completion of the space of smooth functions with respect to the $L^p_k$-norm.  We have the following basic approximation lemma:
\begin{lemma}
Let $\Sigma$ be a compact Riemannian manifold and $B=L^q_l(\Sigma)$. The space $C^\infty(M\times \Sigma)$ is dense in $L^p_k(M;B)$.  
\end{lemma}  
\label{lem12}

\begin{proof}
 Using a partition of unity, it suffices to prove the result when $M=T^d$, where $T^d$ is a $d$-dimensional torus with coordinates $x_i$.  Given $f\in L^p_k(T^d; B)$, we may assume that $f\in C^\infty(T^d; B)$ since such functions are dense in $L^p_k(T^d; B)$.  The Fourier inversion theorem (valid for Banach valued maps) implies that $$f=\sum_\lambda c_\lambda e_\lambda$$ where $c_\lambda \in B$ and $$e_\lambda=e^{ix_1\lambda_1+\dots ix_d\lambda_d}$$ is an eigenfunction of $\Delta_{T^d}=-\sum_{i=1}^d\partial_{x_i}^2$.   We may approximate $f$ in the $C^k$-norm by a finite sum $$\sum_\lambda^{|\lambda|<N} c_\lambda e_\lambda$$  Finally, each of the finitely many coefficients $c_\lambda \in L^q_l(\Sigma)$  may be approximated with respect to the $L^q_l$-norm on $\Sigma$ by $c'_\lambda \in C^\infty(\Sigma)$.   Therefore, $f$ is  arbitrarily close to an element in $C^\infty(T^d\times \Sigma)$.   
\end{proof}

\begin{lemma}
\label{lem1}
We have $L^p(M\times \Sigma)=L^p(M;L^p(\Sigma))$ and $L^p_k(M; L^p(\Sigma))=L^p(\Sigma;L^p_k(M))$.

\end{lemma}
\begin{proof}

The corresponding  norms agree on  $C^\infty(M\times \Sigma)$.  The previous lemma implies that  $C^\infty(M\times \Sigma)$ is dense in all the spaces considered.  Therefore, we get the desired conclusion by taking completions.
\end{proof}
\noindent \textbf{Remark.} As an application of lemma $\ref{lem1}$, note that the equality $$L^p_k(M;L^p(\Sigma))=L^p(\Sigma; L^p_k(M))$$ allows us to define $L^p_k(M;L^p(\Sigma))$ for negative $k$ by taking the right hand side as the definition of $L^p_k(M; L^p(\Sigma))$.  
\begin{lemma}
\label{lem3}
We have $L^p_k(M\times \Sigma)=L^p_k(M; L^p(\Sigma))\cap L^p(M; L^p_k(\Sigma))$.   
\end{lemma}
\begin{proof}
First of all, we note that $$L^p_k(M\times \Sigma)\subset L^p_k(M; L^p(\Sigma))\cap L^p(M; L^p_k(\Sigma))$$ since $C^\infty(M\times \Sigma)$ are dense in all the spaces considered and the norm on $L^p_k(M\times \Sigma)$ controls the norm on $L^p_k(M; L^p(\Sigma))$ and $ L^p(M; L^p_k(\Sigma))$.   To establish the claim we first show that $f\in L^p_k(M; L^p(\Sigma))\cap L^p(M; L^p_k(\Sigma))$ can be simultaneously approximated by a single smooth function.  For this, we assume that $M=T^d$ as in the previous proof.   Let $g_\ep:T^d\rightarrow \Ar$ be a smooth mollification of the Dirac delta function (see \cite{Taylor1} for the details).  Let $$f_\ep=f*g_\ep$$ be the convolution of $f$ with $g_\ep$.  By taking $\ep$ sufficiently small, we may replace $f$ by $f_\ep \in C^\infty(M; L^p(\Sigma))\cap C^\infty(M; L^p_k(\Sigma))$ which approximates $f$ in both norms simultaneously.  Now, we develop $f_\ep$ in a Fourier series $$f_\ep=\sum_\lambda c_\lambda e_\lambda$$ As above, we approximate $f_\ep$ by a finite truncated series $f_\ep=\sum_\lambda^{|\lambda|\leq N} c_\lambda' e_\lambda$ that approximates $f$ arbitrary closely in both the $ L^p_k(M; L^p(\Sigma))$-norm as well as the $L^p(M; L^p_k(\Sigma))$-norm.  \\\\
To prove the claim, we argue as follows.  For $k=1$, the norms on the two sides agree.  For $k=2$, let $f\in  L^p_k(M; L^p(\Sigma))\cap L^p(M; L^p_k(\Sigma))$.  We apply the Laplacian $$\Delta_{M \times \Sigma}f=\Delta_M f+\Delta_\Sigma f \in L^p(M\times \Sigma)$$ By elliptic regularity, (see \cite{Taylor3} and the following section) we have $f\in L^p_2(M\times \Sigma)$.  To prove the general case we need to construct an appropriate analogue of $\Delta$.  For even $k$, we may take  $\Delta_M^{k/2}+\Delta_\Sigma^{k/2}$.  This is an elliptic differential operator of order $k/2$ and thus we may conclude that $f\in L^p_k(M\times \Sigma)$.  For odd $k$, one has a variant of this argument using Dirac operators.    
\end{proof}

Let $M^+$ be compact Riemannian manifold with boundary $\partial M^+$.  We will assume for simplicity that $M^+$ is isometric to a product $[0,1)\times \partial M^+$ near the boundary.    We would like to generalize the previous results to this case.

\begin{lemma}
  $L^p_k(M^+\times \Sigma)=L^p_k(\Sigma; L^p(M^+))\cap L^p(\Sigma; L^p_k(M^+))$.
\end{lemma}
\begin{proof}
Given $f\in L^p_k(\Sigma; L^p(M^+))\cap L^p(\Sigma; L^p_k(M^+))$, we use the extension lemma $\ref{lem4}$ to construct 
$$E(f)\in L^p_k(\Sigma; L^p(M))\cap L^p(\Sigma; L^p_k(M))$$  By lemma $\ref{lem3}$, we have $E(f)\in L^p_k(M\times \Sigma)$.  Restricting to $M^+\times \Sigma$ gives the desired result.

\end{proof}

\section{Linear Elliptic Estimates}

\subsection{Elliptic Operators on Closed Manifolds}
\label{SecReg}
We now briefly review regularity theory for elliptic differential operators on closed manifolds.  The proofs of all the results can be found in \cite{Taylor3}.  \\\\
Let $M$ be a closed Riemannian manifold and let $$D:\Gamma(E)\rightarrow \Gamma(F)$$ be a differential operator acting on sections of Hermitian vector bundles over $M$.  We assume that $D$ has smooth coefficients.  Let $L^p_k(\Gamma(E))$ stand for the completion of the space of smooth sections of $E$ with respect to the norm from equation $\ref{eq7}$.  If $D$ has order $m$, it extends to a continuous map $L^p_k(\Gamma(E))\rightarrow L^p_{k-m}(\Gamma(F))$ for all real $k$ and $p>1$. As before, we will often drop $\Gamma(E)$ from the notation.\\\\
 Let us assume that $D$ is elliptic.  By definition, this means that the principal symbol of $D$ is invertible (see \cite{Taylor1}).  We have the following fundamental result:
 \begin{theorem}
 \label{thmellreg}
 For $f\in L^p_k(\Gamma(E))$, we have $$||f||_{L^p_k}\leq C_{p,k}(||Df||_{L^p_{k-m}}+||f||_{L^p_{k-1}})$$
 Furthermore, if $f\in L^p_k$ and $D(f)\in L^{p'}_{k'}$ we have that $f\in L^{p'}_{k'+m}$.   
 \end{theorem}
Here are some examples of elliptic operators that occur in this work:
\begin{itemize}
\item Given a Hermitian connection $\nabla$ on a vector bundle, we may form the connection Laplacian $\nabla^* \circ \nabla$.  This is an elliptic operator of order 2.
\item  Given a Hermitian connection $\nabla$ on a vector bundle $E$, we may form the first order operator $d_\nabla+d_\nabla^*$ acting on $\Lambda^*(E)=\Lambda^*(M)\otimes \Gamma(E)$.  
\item Given a holomorphic vector bundle $E$ on a Riemann surface, we have the operator $$\delbar: \Gamma(E)\rightarrow \Gamma(E\otimes K^{-1})$$ where $K^{-1}$ is the anticanonical bundle.  
\item If $dim(M)=4$, we have the first order elliptic operator (see \cite{DK} for a general discussion) $$d_\nabla^*+d^+_{\nabla}:\Lambda^1(E) \rightarrow \Lambda^0(E)\oplus \Lambda^+(E)$$  Here $\Lambda^+$ is the bundle of self-dual 2-forms.  
\end{itemize}

\subsection{Dirichlet and Neumann Problem}
\label{Dirichlet}
We discuss the regularity theory for the Dirichlet and Neumann boundary value problems with special emphasis on weak solutions.  Our goal is to explain how to reduce the various regularity statements to the interior cases.  There are many alternative treatments of this material (see \cite{Kat1}).    \\\\
Let $M^+$ be compact Riemannian manifold with boundary $\partial M^+$ and double $$M=M^-\cup_{\partial M^+} M^+$$  We will focus on the $L^p_k$-regularity theory for the case $k\leq 2$ as it is less standard. Let us begin by discussing the Dirichlet problem.  As always, we assume $p>1$.     
\begin{definition}
Let $f\in L^p(M^+)$ and $g\in L^1(M^+)$.  We say that $f$ is a \textbf{weak solution} to 
\begin{equation}
\Delta f=g  \text{ and } f_{|\partial M^+}=0
\end{equation}
 if we have 
\begin{equation}
\langle f, \Delta h \rangle_{M^+} =\langle g,  h \rangle_{M^+}
\end{equation} for all smooth $h\in C^\infty(M^+)$ that vanish on $\partial M^+$.  
\end{definition}
Note that if $f\in L^p_2(M^+)$, this condition coincides with the usual definition $$\Delta f=g \text{ and }f_{|\partial M^+}=0$$  This follows directly from Green's formula 

\begin{equation}
\langle \Delta u,v \rangle_{M^+} - \langle u,\Delta v \rangle_{M^+}  = \int_{\partial M^+} u \partial_\nu v-v\partial_\nu u 
\end{equation}  
where $\partial_\nu$ is the outward normal derivative at the boundary.  
Here is the basic regularity result:
\begin{lemma}
\label{lem9}
Let $f\in L^p$ be a weak solution to $\Delta f=g$, $f_{|\partial M^+}=0$.  If $g\in L^p$, then $f\in L^p_2$.  
\end{lemma}
\begin{proof}
The basic idea it to extend $f$ and $g$ to the double $M$ and then apply elliptic regularity for the closed manifold $M$.  We extend $f$ to $\tilde {f}\in L^p(M)$ by taking $-f$ on the $M^-$ piece of $M$. Similarly, we extend $g$ to $\tilde{g}\in L^p(M)$ by taking $-g$ on $M^-$.   
We need to show that $$ \langle \tilde{f}, \Delta h \rangle_M = \langle \tilde{g}, h \rangle_M $$ for all smooth $h$ on ${M}$. Given such a test function $h$, decompose $h$ as $$h=h_s+h_a$$ where $h_s$ is symmetric with respect to the reflection across $\partial M^+$ and $h_a$ is antisymmetric.  Since $\tilde{g}$ and $\tilde{f}$ are antisymmetric, we have $$\langle \tilde{g}, h_s \rangle_M =\langle \tilde{f}, h_s \rangle_M =0$$  Thus, we may assume $h=h_a$.  In particular, we may assume that $h$ vanishes on $\partial M$ and  $$\langle \tilde{g}, h \rangle_M =2 \langle \tilde{g}, h \rangle_{M^+}$$ while $$\langle \tilde{f}, \Delta h \rangle_M =2 \langle \tilde{f}, \Delta h \rangle_{M^+}$$  By hypothesis, these two integrals are equal.  Now, we may apply theorem $\ref{thmellreg}$ to deduce that $\tilde{f} \in L^p_2(M)$ and thus, by restriction, $f\in L^p_2(M^+)$.
\end{proof}
We now state the regularity results for higher norms:
\begin{lemma}
\label{lemhigherDir}
Let $k\geq 2$, $p>1$.  Let $f\in L^p$ be a weak solution to $\Delta f=g$, $f_{|\partial M^+}=0$.  If $g\in L^p_k$, then $f\in L^p_{k+2}$. 

\end{lemma}
\begin{proof}
This follows by induction from the previous lemma as explained for instance in \cite{Taylor1}.
\end{proof}
We turn now to the corresponding Neumann boundary value problem. Let $\partial_\nu f$ be outward normal derivative on $\partial M^+$.  
\begin{definition}
Let $f\in L^p(M^+)$ and $g\in L^1(M^+)$.  We say that $f$ is a weak solution to 
\begin{equation}
\Delta f=g  \text{ and } \partial_\nu f=0
\end{equation}
 if we have 
\begin{equation}
\langle f, \Delta h \rangle =\langle g,  h \rangle
\end{equation} for all smooth $h\in C^\infty(M^+)$ with $\partial_\nu h=0$.  
\end{definition}  
Here is the corresponding regularity result:
\begin{lemma}
\label{lemNeum1}
Let $f\in L^p$ be a weak solution to $\Delta f=g$, $\partial_\nu f=0$.  If $g\in L^p$, then $f\in L^p_2$.  
\end{lemma}
\begin{proof}
The argument is virtually identical to lemma $\ref{lem9}$.  The main difference is that now one uses the symmetric extension of $f$ to $M$ instead of the antisymmetric extension of lemma $\ref{lem9}$.  
\end{proof}
We now state the regularity results for higher norms:
\begin{lemma}
\label{lemhigherDNeu}
Let $k\geq 2$, $p>1$.  Let $f\in L^p$ be a weak solution to $\Delta f=g$,  $\partial_\nu f=0$.  If $g\in L^p_k$, then $f\in L^p_{k+2}$. 
\end{lemma}
\begin{proof}
This follows by induction from the previous lemma as explained for instance in \cite{Taylor1}.
\end{proof}
We consider now the inhomogeneous case of these equations.  We focus on the Neumann case as is it less standard.  Here is the fundamental result of Nirenburg that we will use (see \cite{Kat1} for a proof):
\begin{theorem}
\label{thm4}
Let $p>1$ and $k \geq 1$.  There exists a continuous map 
\begin{equation}
T:L^p_k(M^+)\rightarrow L^p_{k+1}(M^+) 
\end{equation}
 such that for every $f\in L^p_k(M^+)$ we have
 \begin{equation}
  \partial_\nu T(f)=f_{|\partial M^+}
 \end{equation} 
\end{theorem}
\begin{proof}
This theorem is essentially contained in  \cite{Kat1} and thus we restrict ourselves to a brief sketch.  Let us first consider the local construction.  Let $$H^2=\{ (s,t)\in \Ar^2| s \geq 0\}$$ be the half space and consider $f\in L^p_1(H^2)$ with support on the standard unit disk $D_1$.  Let $K(s,t)$ be defined by $$K(s,t)=\frac{\ln(s^2+t^2)}{2\pi}$$  Let  $$g(s,t)=-\int_{\Ar}K(s,t-\tau)f_{|\Ar}(\tau) d\tau $$   By construction (see \cite{Kat1}), $$\partial_s g(0,t)=-f_{|\Ar}$$ and $$||g||_{L^p_2}\leq C ||f||_{L^p_1} $$   
Now, we will obtain a compactly supported modification of $g$  using a bump function.  Let $$\rho_1:\Ar\rightarrow \Ar$$ be a bump with support in $D_2$ such that $\rho_1=1$ on $D_1$.  Let $$\rho_2(s,t)=\rho_1(s)\rho_1(t)$$  By construction, $\partial_s \rho_2 =0$ on $\partial H^2$.  The compactly supported $\rho_2 g$ has $$\partial_s \rho_2 g =-f$$ on $\partial H^2$ as desired. To obtain the global operator $T$, we proceed as above using a partition of unity near $\partial M^+$.

\end{proof}

\begin{definition}
Let $f\in L^p(M^+)$, $r\in L^p_1(M^+)$   and $g\in L^1(M^+)$.  We say that $f$ is a weak solution to 
\begin{equation}
\Delta f=g  \text{ and } \partial_\nu f=r_{|\partial M^+}
\end{equation}
 if we have 
\begin{equation}
\langle f, \Delta h \rangle_{M^+} =\langle g,  h \rangle_{M^+}+\langle r,  h \rangle_{\partial M^+}
\end{equation} for all smooth $h\in C^\infty(M^+)$ with $\partial_\nu h=0$.  
\end{definition}
We summarize the results for the inhomogenous Neumann problem with the following:
\begin{lemma}
\label{leminNeum}
Let $f\in L^p$ be a weak solution to $\Delta f=g$, $\partial_\nu f=r_{|\partial M^+}$.  If $g\in L^p$, $r\in L^p_1$ then $f\in L^p_2$.  
\end{lemma}
\begin{proof}
By theorem $\ref{thm4}$, we may take $u\in L^p_2(M^+)$ with $\partial_\nu u=f_{|\partial M^+}$.  Let $f'=f-u$.  By construction, $f'$ satisfies the homogeneous (weak) Neumann problem: 
\begin{equation}
\Delta f'=g-\Delta u  \text{ and } \partial_\nu f'=0
\end{equation} 
Since the $L^p_2$-norm of $u$ is controlled by the $L^p_1$-norm of $r$, we may apply lemma $\ref{lemNeum1}$ to deduce $f\in L^p_2$ as desired.
\end{proof}
We also have a version of this result for higher Sobolev norms:
\begin{lemma}
Assume $k\geq 1$.  Let $f\in L^p_k$ be a solution to $\Delta f=g$, $\partial_\nu f=r_{|\partial M^+}$.  If $g\in L^p_k$, $r\in L^p_{k+1}$ then $f\in L^p_{2+k}$.  
\end{lemma}
The proof is a straightforward inductive argument using lemma $\ref{leminNeum}$  as the base case.  For details, see \cite{Taylor1}.

\subsection{Elliptic Theory for Banach Space Valued Functions}
\label{BanahEst}
We now extend the regularity results discussed in the previous sections to the setting of Banach valued maps.  As before, let $M$ be a closed manifold of dimension $d$ and let  $B=L^p(\Sigma)$.
Much of the regularity theory for elliptic operators carries over to the setting of Banach valued maps.  Here is the basic result:
\begin{lemma}
\label{lemregBanach}
 For $k\geq 2$, we have an isomorphism $$\Delta+1:L^p_k(M; B)\rightarrow L^p_{k-2}(M; B)$$  In particular, 
 $$||f||_{L^p_{k}}\leq C_{p.k}(||\Delta f||_{L^p_{k-2}}+||f||_{L^p_{k-2}})$$ 
\end{lemma}
\begin{proof}
Note that  $$\Delta+1:L^p_k(M)\rightarrow L^p_{k-2}(M)$$ is an isomorphism for any $k$.  Since by lemma $\ref{lem1}$ $$L^p_k(M;B)=L^p(\Sigma; L^p_k(M))$$ we have the induced isomorphism $$\Delta+1:L^p(\Sigma; L^p_k(M))\rightarrow L^p(\Sigma; L^p_{k-2}(M))$$  
\end{proof}
Along with the elliptic estimate we have a regularity result.  Given $f, g \in L^p(M; B)$, we say that $f$ is a weak solution to $$\Delta f=g$$ if for all $h\in C^\infty(M\times \Sigma)$, we have $$\langle f , \Delta h\rangle_M = \langle g ,  h\rangle_M$$   Note that the pairings are well defined since $f,g \in L^p(M\times \Sigma)$. 
\begin{lemma}
\label{BanReg}
 Let $k\geq 0$.  Assume $f\in L^p(M; B)$ and $\Delta f\in L^p_{k}(M; B)$.  We have $f\in L^p_{k+2}(M; B)$.  
\end{lemma}
\begin{proof}
We address the case $k=0$ as the other cases are similar.  By the previous lemma, we may take $u \in L^p_{2}(M; B)$ such that $(\Delta +1)u=(\Delta+1)f \in L^p$.  Thus, by considering $f-u$ we may assume that $$(\Delta+1) f=0$$ in the weak sense.  In particular, this implies that $$\langle f, (\Delta +1) h \rangle _{M}=0$$ for all smooth $h$.  Since $(\Delta+1)h$ is dense in $L^p(M;B)$ by lemma $\ref{lem1}$ we conclude that $f=0$ as desired.  
\end{proof}   
We now discuss how the elliptic estimates for the Dirichlet and Neumann problem carry over to Banach space valued functions.  Just as in the closed case, the proofs of the various regularity results can be reduced to their finite dimensional counterparts.  
\begin{definition}
Let $f\in L^p(M^+; B)$ and $g\in L^1(M^+; B)$.  We say that $f$ is a weak solution to 
\begin{equation}
\Delta f=g  \text{ and } f_{|\partial M^+}=0
\end{equation}
 if we have 
\begin{equation}
\langle f, \Delta h \rangle_{M^+} =\langle g,  h \rangle_{M^+}
\end{equation} for all smooth $h\in C^\infty(M^+\times \Sigma)$ that vanish on $\partial M^+ \times \Sigma$.  
\end{definition}
  Here is the corresponding regularity result:
\begin{lemma}
\label{lemDBanach}
Let $f\in L^p$ be a weak solution to $\Delta f=g$, $f_{|\partial M^+}=0$.  If $g\in L^p$, then $f\in L^p_2$.  
\end{lemma}
\begin{proof}
The reduction to the interior case is identical to the proof of lemma $\ref{lem9}$, this time using lemma $\ref{BanReg}$.
\end{proof}
We now state the regularity results for higher norms:
\begin{lemma}
Let $k\geq 2$, $p>1$.  Let $f\in L^p$ be a weak solution to $\Delta f=g$, $f_{|\partial M^+}=0$.  If $g\in L^p_k$, then $f\in L^p_{k+2}$. 
\end{lemma}
\begin{proof}
This follows by induction from the previous lemma as in case of lemma $\ref{lemhigherDir}$.
\end{proof}

We turn now to the corresponding Neumann boundary value problem. Let $\partial_\nu f$ be outward normal derivative on $\partial M^+$.  
\begin{definition}
Let $f\in L^p(M^+; B)$ and $g\in L^1(M^+; B)$.  We say that $f$ is a weak solution to 
\begin{equation}
\Delta f=g  \text{ and } \partial_\nu f=0
\end{equation}
 if we have 
\begin{equation}
\langle f, \Delta h \rangle_{M^+} =\langle g,  h \rangle_{M^+}
\end{equation} for all smooth $h\in C^\infty(M^+\times \Sigma)$ with $\partial_\nu h=0$.  
\end{definition}  
Here is the corresponding regularity result:
\begin{lemma}
Let $f\in L^p$ be a weak solution to $\Delta f=g$, $\partial_\nu f=0$.  If $g\in L^p$, then $f\in L^p_2$.  
\end{lemma}
\begin{proof}
The argument is similar to lemma $\ref{lemDBanach}$.  The main difference is that now one uses the symmetric extension of $f$ to $M$ instead of the antisymmetric extension in lemma $\ref{lemDBanach}$.  
\end{proof}
We now state the regularity results for higher norms:
\begin{lemma}
Let $k\geq 2$, $p>1$.  Let $f\in L^p$ be a weak solution to $\Delta f=g$,  $\partial_\nu f=0$.  If $g\in L^p_k$, then $f\in L^p_{k+2}$. 
\end{lemma}
\begin{proof}
This follows by induction from the previous lemma as in case of lemma $\ref{lemhigherDNeu}$.
\end{proof}
We consider now the inhomogeneous case of these equations.  As before, we focus on the Neumann case.  
\begin{lemma}
Let $p>1$ and $k \geq 1$.  There exists a continuous map 
\begin{equation}
T:L^p_k(M^+; B)\rightarrow L^p_{k+1}(M^+; B) 
\end{equation}
such that for every $f\in L^p_k(M^+; B)$ we have
\begin{equation}
  \partial_\nu T(f)=f_{|\partial M^+}
 \end{equation} 
\end{lemma}
\begin{proof}
Since  $L^p_k(M^+; L^p(\Sigma))=L^p(\Sigma; L^p_k(M^+))$, we define $$T:L^p(\Sigma; L^p_k(M^+))\rightarrow L^p(\Sigma; L^p_{k+1}(M^+))$$ using theorem $\ref{thm4}$. 
\end{proof}
\begin{definition}
Let $f\in L^p(M^+; B)$, $r\in L^p_1(M^+; B)$   and $g\in L^1(M^+; B)$.  We say that $f$ is a weak solution to 
\begin{equation}
\Delta f=g  \text{ and } \partial_\nu f=r_{|\partial M^+}
\end{equation}
 if we have 
\begin{equation}
\langle f, \Delta h \rangle_{M^+} =\langle g,  h \rangle_{M^+}+\langle r,  h \rangle_{\partial M^+}
\end{equation} for all smooth $h\in C^\infty(M^+\times \Sigma)$ with $\partial_\nu h=0$.  
\end{definition}
We summarize the results for the inhomogenous Neumann problem with the following:
\begin{lemma}
Assume $k\geq 0$.  Let $f\in L^p_k$ be a solution to $\Delta f=g$, $\partial_\nu f=r_{|\partial M^+}$.  If $g\in L^p_k$, $r\in L^p_{k+1}$ then $f\in L^p_{2+k}$.  
\end{lemma}
\begin{proof}
The proof is identical to the case when $B$ is finite dimensional.
\end{proof}
Finally, we point out that the previous lemmas have an extension to the case $B=L^p(\Sigma)\oplus \Ar^N$.  In this case, we separate the finite dimensional part and treat it using the methods discussed above. \\\\
We now turn to regularity results for first order elliptic operators.  Consider some first order elliptic differential operator $$D:\Gamma(E)\rightarrow \Gamma(F)$$ acting on sections of bundles over $M$.  If we let $B_p=L^p(\Sigma)$, we get an induced operator 
$$D:L^p_k(M; B_p\otimes E)\rightarrow L^p_{k-1}(M; B_p \otimes F)$$ for all $k\geq 0$.  We have:

\begin{lemma}
Let $k\geq 1$.  Given $f\in L^p_k(M; B_p\otimes E)$, we have $$||f||_{L^p_{k}}\leq C(||Df||_{L^p_{k-1}}+||f||_{L^p_{k-1}})$$  Assume, $ Df\in L^q_{k-1}(M; B_q)$ and $f\in L^q_{k-1}(M; B_q)$ with $q>p$.  We have $f\in L^q_{k}(M; B_q)$.   

\end{lemma}
\label{lem18}
\begin{proof}
By considering $D'=D\oplus D^*$, we may as well assume that $D$ is self-adjoint.  Furthermore, for a generic choice of $c\in \Ar$, $D+c$ is invertible  as a map $L^p_k\rightarrow L^p_{k-1}$.  As before, we identify $L^p_k(M;B_p\otimes E)=L^p(\Sigma; L^p_k(\Gamma(E)))$ and obtain the isomorphism 

$$D+c: L^p_k(M; B_p\otimes E)\rightarrow  L^p_{k-1}(M; B_p\otimes E)$$  The rest of the proof is identical to the proof of lemma $\ref{lemregBanach}$.
\end{proof}   
\noindent  \textbf{Remark.}  Let $B^0_p=L^p(N)\times \Ar^n$ and let $B_p=B^0_p\oplus B^0_p$ with complex structure $J$ that maps $(a,b)$ to $(-b,a)$.  We obtain the corresponding $\delbar$-operator $$L^p_k(T^2; B_p)\rightarrow L^p_{k-1}(T^2; B_p)$$ for all $k\geq 1$.  This example is a particular case of the previous construction where we take $D=\delbar$ and $E=\C^n$.  Note that $\delbar+1$ is an isomorphism on $L^p_k(T^2; B_p)$.  \\\\
\textbf{Remark.}  Let $D_R\subset \C$ be the closed disk of radius $R$. For any $k\geq 1$, let $L^p_{k;\Lag}(D_R; \C)$ be subspace of $L^p_{k}(D_R; \C)$ with imaginary values on $\partial D_R$.   Consider the operator $$\delbar: L^p_{k;\Lag}(D_R; \C)\rightarrow L^p_{k-1}(D_R;\C)$$ as discussed in \cite{MS}.  This operator is surjective with left inverse denoted by $T$.  Note that $$B_p=B^0_p\otimes_{\Ar}\C$$ and $ B^0_p\otimes_{\Ar}(i\Ar)$ is Lagrangian.   This implies that the operator $$\delbar: L^p_{k;\Lag}(D_R; B_p)\rightarrow L^p_{k-1}(D_R;B_p)$$ is surjective with a left inverse induced by $T$.

\section{The Moduli Space of Projectively Flat Connections on a Surface}
\label{secRep}
\subsection{Basic Construction}
Consider a compact Riemann surface $\Sigma$ with complex structure $j_\Sigma$ and Kahler metric $g_\Sigma$.  Let $E\rightarrow \Sigma$ be a 2-dimensional complex vector bundle with odd $$c_1(E)\in H^2(\Sigma; \mathbb{Z}) \cong \mathbb{Z}$$  Let $V \rightarrow \Sigma$ be the bundle of traceless endomorphisms of $E$.  Note that $V$ has structure group $SO(3)$.   We fix once and for all a unitary connection $\alpha_{det}$ on $det(E)$.  Once this choice of $\alpha_{det}$ is made, we have an identification between $SO(3)$-connections on $V$ and $U(2)$-connections on $E$ that induce $\alpha_{det}$ on $det(E)$.\\\\
Let $\GauS_E$ be the group of $U(2)$-gauge transformations of $E$ that descend to the identity on $det(E)$.  $\GauS$ has a natural induced action on the connections on $V$ and we let $\GauS\subset \GauS_V$ be its image.  We may identify the action of $\GauS$ with the action of $SO(3)$-gauge transformations on $V$ that lift to gauge transformations of $E$.      

\begin{definition}
$\Aff_{p,k}$ be the affine space of$SO(3)$-connections on $V$ completed with respect to the $L^p_k$-norm .  $\Aff_{p,k}$ is a space modeled on the space of traceless endomorhisms of $V$, $\Omega^0(\Sigma; \liea)$.     
\end{definition}

\begin{definition}
Let $\GauS_{p,k+1}$ be completion of $\GauS$ with respect to the $L^p_{k+1}$-topology.  $\GauS_{p,k+1}$ is a fibre bundle with fibre $SO(3)$.   
\end{definition}
We will assume that $p(k+1)>2$.   
$\GauS$ acts on $\Aff$ by $$g^*(\alpha)=\alpha+g^{-1}d_{\alpha}g$$
Following Atiyah and Bott \cite{AB}, we note that this action is Hamiltonian with moment map $$\Moment: \Aff_{p,k} \rightarrow \Omega^0_{p,k}(\Sigma; \liea)$$ given by $\Moment(\alpha)=*_2F_\alpha$, where $F_\alpha$ is the curvature of $\alpha$.  
\begin{lemma}
$\Moment$ is a smooth map $L^p_k\rightarrow L^p_{k-1}$ for $(k+1)p>2$.  Furthermore, $0\in \Omega^0(\Sigma; \liea)$ is a regular value of $\Moment$.   
\end{lemma}
\begin{proof}
Let us assume that $k=0$, as the other cases are easier.   We examine the nonlinear part of the moment map that sends $\alpha$ to $\alpha \wedge \alpha$. Let $\frac{1}{p^*}+ \frac{1}{p}=1$.  In dimension 2, we have the embedding $L^{p^*}_1\rightarrow L^{\frac{2p^*}{2-p^*}}$ and $$L^p\cdot L^p \subset L^{p/2}$$  Therefore, $\alpha \wedge \alpha$ defines an element of $L^p_{-1}$ by the pairing $$L^p\cdot L^p \cdot L^{\frac{2p^*}{2-p^*}} \subset L^{1}$$ since $$\frac{2}{p}+\frac{2-p^*}{2p^*}=\frac{1}{p}+\frac{1}{2}<1$$  
If $\Moment(\alpha)=0$ then $\alpha$ is flat and thus must be irreducible.  This implies that $$D\Moment_\alpha=*_2 d_\alpha$$ is surjective.  
\end{proof}
Let $\Flat \subset \Aff$ be the set of projectively flat connections.  
\begin{definition}
Let $\Mod=\Moment^{-1}(0)/\GauS=\Aff // \GauS$
\end{definition}
In fact, $\Mod$ is a compact Kahler manifold (see \cite{AB}).  $\Mod$ has a concrete description in terms of representations of $\pi_1(\Sigma)$.  Pick a point $p\in \Sigma$.  The space $\Mod$ is the space of representations $\pi_1(\Sigma-p)\rightarrow SU(2)$ that have holonomy $-I$ around $p$, modulo conjugation by $SU(2)$.  If we pick a standard homology basis  $\{\alpha_i \}_{i=1}^g,\{\beta_i \}_{i=1}^g$ , we may identify $\Mod$ with the space 
\begin{equation}
\{ g_i\in SU(2), h_i\in SU(2)|\Pi_{i=1}^g[g_i,h_i]=-I \}/SU(2)
\end{equation}
We now describe a convenient local parametrization for $\Flat$ and $\Mod$.  Fix a flat connection $\alpha_0$.  Consider the map 
\begin{equation}
\label{Feq1}
\tilde{F}: \Aff_{p,k} \rightarrow \Omega^0_{p,k-1}(\Sigma; \liea)\oplus  \Omega^0_{p,k-1}(\Sigma; \liea) \oplus \harmo^1_{\alpha_0}(\Sigma)
\end{equation}
where $\tilde{F}=\Moment\oplus d_{\alpha_0}^*\oplus \Pi_0$ and $\Pi_0$ is the projection to the finite dimensional space of harmonic $(d_{\alpha}^*+d_{\alpha})$-forms.  This map is a diffeomorphism near $\alpha_0$ and gives us a local parametrization of $\Mod$ and $\Flat$.  The linearization of $\tilde{F}$ at a point $\alpha\in \Aff$ gives us the map $$D\tilde{F}_\alpha:\Omega^1_{p,k}(\Sigma; \liea)\rightarrow \Omega^0_{p,k-1}(\Sigma; \liea)\oplus \Omega^0_{p,k-1}(\Sigma; \liea) \oplus \harmo^1_{\alpha_0}(\Sigma)$$    
An important technical point that will be useful in establishing regularity is that for each $\alpha\in L^p$, $D\tilde{F}_\alpha$ extends to a continuous map $$L^q\rightarrow L^q_{-1}$$ for all $q \geq p^*$. This is immediate for $\Pi_0$ and $d_0^*$ and thus we need only treat the component given by $\mu$.  Since $$D\mu_a(v)=d_0 v+2[\alpha,v]$$ the continuity of the extension is a consequence of the fact that $$L^p\cdot L^q \rightarrow L^q_{-1}$$
To justify this embedding we argue as follows.  By definition, $L^{q}_{-1}=(L^{q^*}_{1})^*$ where $1/q+1/(q^*)=1$.  Thus, to justify the embedding we need to establish that $$L^p\cdot L^q\cdot L^{q^*}_1\rightarrow L^1$$  This follows by direct computation.  Alternatively, recasting this claim in the notation of section $\ref{Sobolev}$, we must establish $$L^{1/\overline{p}}\cdot L^{1/\overline{q}} \cdot L^{1/(1-1/\overline{q})}_1 \rightarrow L^1$$ where ${\overline{p}}=1/p$ and ${\overline{q}}=1/q$.  First assume $q>2$.  We have $ L^{1/(1-\overline{q})}_1\rightarrow  L^{1/(1-\overline{q}-1/2)}$ and
$$L^{1/\overline{p}}\cdot L^{1/\overline{q}}\cdot  L^{1/(1-\overline{q})}_1 \rightarrow L^{1/r}$$  where $r=\overline{p}+\overline{q}+1-\overline{q}-1/2=\overline{p}+1/2<1 $ since $p>1/2$.   Now take $q < 2$.  In this case $L^{q^*}_1\subset L^\infty$ and thus $$L^{1/\overline{p}}\cdot L^{1/\overline{q}} \subset L^{1/(\overline{p}+\overline{q})}\subset L^1$$ as long as $\overline{p}+\overline{q}\leq 1$.     The case $q=2$ is similar.

We also note that the map $$D\tilde{F}:\Omega_{p,0}^0(\Sigma)\rightarrow End(\Omega_{q,0}^0(\Sigma; \liea)_{q,0},\Omega_{q,-1}^0(\Sigma; \liea)\oplus \Omega_{q,-1}^0(\Sigma; \liea) \oplus \harmo^1_{\alpha_0}(\Sigma))$$ is a smooth map of $\alpha$.  Indeed, $D\tilde{F}$ is linear and continuous in $\alpha$ and therefore smooth.  To abstract the situation, let $V_p=L^p(\Sigma)$ be the space of $L^p$-sections of some bundle over $\Sigma$.  We say that a map $$T:V_p\rightarrow End(V_p,V_p)$$ is compatible with the underlying $L^q$-structure if it extends to a smooth map    $$T:V_p\rightarrow End(V_q,V_q)$$ for $p^*\leq q $.    
\label{sectLq}         
\subsection{The Canonical Lagrangian}
\label{Canlag}
Given a symplectic manifold $(M,\omega)$ with a Hamiltonian group action $G$ and corresponding moment map $$\mu:M\rightarrow Lie(G)$$ one may form the symplectic reduction at $0$ by $$M//G=\mu^{-1}(0)/G$$ Provided $G$ acts freely on $\mu^{-1}(0)$, $(M//G,\omega_{M//G})$ inherits a symplectic structure from $M$.  Let $(M//G) ^-$ denote the symplectic manifold with the opposite form. There is a canonical Lagrangian $$\Lag=\mu^{-1}(0)\subset (M//G) ^-\times M$$ defined as the set of pairs $([m],m)$.  Here, $[m]=mG$ denotes the orbit of $m$.  This general construction applies to the case of interest where $M=\Aff$ and $G=\GauS$.  
\begin{definition}
Let $\Lag_{p,k} \subset \Mod\times \Aff_{p,k}$ be the set of pairs $([\alpha],\alpha)$ where $\alpha$ is a flat connection on $\Sigma$.  
\end{definition}   
The goal of the present section is to construct a convenient choice of local charts for $\Lag$. \\\\
Let $U \subset  \harmo^1_{\alpha_0}$ be an open ball around the origin.  If $U$ is sufficiently small, we have a local diffeomorphism 
\begin{equation}
\label{lineqrep}
f:U \rightarrow \Mod
\end{equation}
 around a point $[\alpha_0]\in \Mod$.  We will let $j_0$ denote the induced complex structure on $U$.  Let $\Flat \subset \Aff$ denote the submanifold of flat connections.  We have that $\tilde{F}$ from equation $(\ref{Feq1})$ gives us a local identification of $\Flat$ with an small open ball $$V_{p,k}\subset \Omega^0_{p,k-1}(\Sigma)\oplus \harmo^1_{\alpha_0} $$ Let  $J=(j_0,*_\Sigma)$ denote the product complex structure on $U\times \Aff_{p,k}$. If we restrict the inverse of $F$ to $V_{p,k}$, we obtain a local embedding of the canonical Lagrangian $$G:V_{p,k} \rightarrow  U \times \Aff_{p,k}$$ where $G(v)=([F^{-1}(v)],F^{-1}(v))$.  We may extend $G$ to obtain a local chart for $  U \times \Aff_{p,k}$ by the map $$H:V_{p,k}   \oplus V_{p,k} \rightarrow U \times \Aff_{p,k} $$ that sends $$(u,v) \mapsto  G(u)+J(G(v))$$   This provides a local identification of $\Lag$ with $(u,0)$ and that along $\Lag$ the induced complex structure sends $(u,v)$ to $(-v,u)$.  As in the previous section, $DH$ preserves the $L^q$-structure on $V_{p,k} \oplus V_{p,k}$ in the case $k=0$ and extends to a smooth mapping between these spaces.

\section{A Priori Estimates}
\label{apriori}
\subsection{ASD Equation/J-Curve Equations}
Let us setup some basic conventions.  Let $X$ be a smooth oriented 4-manifold with metric $g_X$.  In this work we will be interested in $X \subset \C \times \Sigma$ with the product metric.   We will use $(x,y)$ for the local coordinates on $\Sigma$ and $(s,t)$ as coordinates on $\C$.   We have the Hodge star operator in dimension 4:
\begin{gather*}
*_4(dxdy)=dsdt\\
*_4(\alpha dt)=(*_2\alpha)ds\\
*_4(\alpha ds)=-(*_2\alpha)dt
\end{gather*}

 If $A$ is an $SO(3)$-connection, we have the gauge group action:
\begin{equation}
g^*(\nabla
_A)s=g^{-1}\nabla_A(gs)=g^{-1}(d+A)gs)=ds+g^{-1}dgs+g^{-1}Ags=\nabla_A s+g^{-1}\nabla_A g
\end{equation}
On $\C \times \Sigma$, we can decompose our connection $A$ as  $$A=\alpha+\phi ds+\psi dt$$ and its curvature as 
\begin{equation}
\begin{split}
F_A= & F_\alpha-\partial_t\alpha dt-\partial_s \alpha ds  +(d_2 \phi +[\alpha,\phi])ds+(d_2 \psi +[\alpha,\psi])dt \\
&+ (\partial_s\psi  -\partial_t \phi +[\phi,\psi])ds dt
\end{split}
\end{equation}
The anti-self-duality (ASD) equation 
\begin{equation}
\label{eq1}
F_A+*_4F_A=0 
\end{equation}
 becomes the pair of equations

\begin{equation}
\label{ASDcomp}
\begin{split}
\partial_s \alpha -d_\alpha \phi +*_{2}(\partial_t \alpha -d_\alpha \psi)= &0 \\
*_2F_\alpha+\partial_s\psi -\partial_t \phi+ [\phi,\psi]=&0
\end{split}
\end{equation}

In general, the energy of a connection $A$ is defined as 
\begin{equation}
\label{eq2}
\frac{1}{2}\int_{X}|F_A|^2d\mu_X
\end{equation}
where $d\mu_X$ is the volume element associated to $g_X$. On a closed 4-manifold $X$, the second Chern class is given by the formula 
\begin{equation}
c_2(P)=\frac{1}{8\pi^2}\int_Xtr(F_A^2)=\frac{1}{8\pi^2}\int_Xtr((F^+_A)^2)+tr((F^-_A)^2)==\frac{1}{8\pi^2}\int_X(|F^-_A|^2-|F^+_A|^2)d\mu_X
\end{equation}
where we use the convention that $|D|^2=tr(D^*D)=-tr(D^2)$ for any skew-hermitian endomorphism $D$.  Thus, for an ASD connection $A$, we have 
\begin{equation}
c_2(P)=\frac{1}{8\pi^2}\int_X|F_A|^2d\mu_X
\end{equation}

Let $\Mod=\Mod(\Sigma)$ be the representation variety as in section $\ref{secRep}$ and let let $D$ be the open unit disk.  A holomorphic curve $u:D\rightarrow \Mod$ with $C^0$ small image may be lifted to a map $$\alpha:D\rightarrow \Aff $$ satisfying 
\begin{equation}
\label{jcurveeq}
\begin{split}
\partial_s \alpha -d_\alpha \phi +*_{2}(\partial_t \alpha -d_\alpha \psi)= &0 \\
F_\alpha= &0 \\
d^*_{\alpha_0}(\alpha-\alpha_0)= &0
\end{split}
\end{equation}
where $\alpha_0=\alpha(0)$.
This is a consequence of the inverse function theorem applied to the map from equation $(\ref{lineqrep})$.  It is instructive to compare equation $(\ref{jcurveeq})$ to $(\ref{ASDcomp})$.

\subsection{Matching Boundary Conditions}
\label{MatchB}
Let $B_R(p) \subset \C$ be the closed disk of radius $R$ centered at $p$ and let $$H^+=\{(s,t)\in \C| s \geq 0  \}$$ be the positive half-plane.  We define $D_R(p)$ as $H^+\cap B_R(p)$.  Perhaps it is more natural to use $D_R^+$ as notation.  However, since we will work mostly on the positive half plane we drop the $+$ for simplicity. Let $D_R^-(p)$ be the reflection of this disk in the $t$-axis.   In general, $\partial D_R=I_R\cup S_R$ where $$I_R=\{(0,t)\in D_R \}$$ and $$S_R=\{(s,t)\in D_R(p)| s^2+t^2=R^2\}$$    We will often drop $p$ from the notation when the $p$ does not change in a particular discussion.  The \textbf{interior} of a disk $D_R$ is the set of points with $s^2+t^2 <R^2$ and will be denoted by $\mathring{D}_R$.  \\\\
Consider a holomorphic map $$u: D_R^- \rightarrow \Mod $$ and an ASD connection $A$ on $D_R \times \Sigma$. For each $t\in I_R$, we may restrict $A$ to the slice $(0,t)\times \Sigma$.  This gives us a map $$R_{A}:I_R\rightarrow \Aff(\Sigma)$$
\begin{definition}
The pair $(u,A)$ is said to be matched if at each $t\in I_R$, $u(0,t)=[R_{A}(t)]$ where $[R_{A}(t)]$ denotes the gauge orbit of $R_{A}(t)$.
\end{definition}
Note that this is precisely the condition that $(u(0,t),R_A(t))\in \Lag$ for each $t$.  If $A$ was a holomorphic curve, this would amount to a Lagrangian boundary condition for the pair $$(\tilde{u},A):D_R\rightarrow \Aff(\Sigma)$$ where $\tilde{u}(s,t)=u(-s,t)$.   For convenience, we will refer to the pair $(u,A)$ as defined on $D_R\times \Sigma$.  

\subsection{Statement of the Result}
 Let $f_{\pm}$ be the functions defined by $$f_+(s,t,x,y)=F_A(s,t,x,y)$$ and $$f_-(s,t,x,y)=|d u(-s,t)|$$  We will also make use of $$e:D_R \rightarrow \Ar $$ defined by $e(s,t)=e^+(s,t)+e^-(s,t)$ where $$e^+(s,t)=\frac{1}{2}\int_\Sigma |f_+(s,t,x,y)|^2d\mu_{\Sigma}$$ and $$e^-(s,t)=\frac{1}{2}f_-(s,t)^2$$  
with $$\partial_s e =\int_{\Sigma}\langle f_+ ,\nabla_s f_+\rangle +f_- \partial_sf_-$$ 
The estimates below will keep track of the radius $R$.  We will always assume that all the constants do not depend on $R$ or the choice of functions $f_{\pm}$.  The following result is key in our proof of compactness for matched pairs.  

\begin{theorem}
\label{thm2.1}
There exists $\hbar,C>0$ with the following property.  Let $(u,A)$ be a matched pair on some $D_R \times \Sigma$.  If $$\Energy(u,A)< \hbar$$ then 
$$e(p)\leq C\Energy(u,A) R^{-2}$$
\end{theorem}
The rest of the section is devoted to the proof of this theorem.

\subsection{Weitzenb\"{o}ck Formulae}
Given an $SO(3)$-connection on a 4-manifold $X$, we have the Weitzenb\"{o}ck formula
$$\nabla_A^*\nabla_A F_A+\{F_A, R_X\}+\{F_A,F_A\}=(d_A^*d_A+d_Ad_A^*) F_A$$
where the brackets denote some pointwise multiplication and the term $R_X$ depends only on the metric of $X$.   
See \cite{Lawson} for a detailed discussion.  What is important for our purposes is that the order zero terms are at most quadratic in $F_A$.  
If $F_A$ is ASD, we have that in particular  $F_A$ satisfies the Yang-Mills equation  $$d_A^* F_A=0$$  and therefore 
\begin{equation}
\nabla_A^*\nabla_A F_A=-\{F_A\,R_X\}-\{F_A,F_A\}
\label{eq2.2.0}
\end{equation}

An application of the Weitzenb\"{o}ck formula (see \cite{survey book} for the holomorphic curve case) leads to pointwise estimates:
\begin{equation}
\label{eq2.2.1}
\Delta_4 |f_+|^2 \leq C'(|f_+|^2+|f_+|^3)-|\nabla_A F_A|^2\leq  C'(|f_+|^2+|f_+|^3)
\end{equation}
\begin{equation}
\Delta_2 f_-^2=\Delta_4 f_- ^2 \leq C'(f_-^2+f^{4}_-)- |du|^2 \leq C'(f_-^2+f^{4}_-)
\end{equation}
For a given nonnegative function $g$,  the relation $$\Delta g^2=2g\Delta g -|\nabla g|^2 \leq 2 g\Delta g $$ leads to 
\begin{equation}
\Delta_4 |f_+| \leq C'(|f_+|+|f_+|^2)
\end{equation}
\begin{equation}
\Delta_2 f_-=\Delta_4 f_- \leq C'(f_-+f^{3}_-)
\end{equation}
Integration of ($\ref{eq2.2.1}$) on $\Sigma$ gives
$$\Delta_2\int_\Sigma |f_+|^2d\mu_{\Sigma} \leq C'(\int_\Sigma |f_+|^2d\mu_{\Sigma}+v(s,t)(\int_\Sigma |f_+|^2d\mu_{\Sigma})^{1/2})$$
where $$v(s,t)=(\int_\Sigma |f_+|^{4}(s,t)d\mu_{\Sigma})^{1/2}$$  is obtained from the Cauchy-Schwartz inequality:  
$$\int_\Sigma |f_+|^3d\mu_{\Sigma}\leq (\int_\Sigma |f_+|^2d\mu_{\Sigma})^{1/2} (\int_\Sigma |f_+|^4d\mu_{\Sigma})^{1/2} $$
Thus, we have 
\begin{equation}
\label{eq2.2.2}
\Delta_2 e\leq 2C'(e+e^2+e^{1/2}v)
\end{equation}

\subsection{Normal Estimates}
\label{secnormal}
So far, we have not used the matching boundary conditions and thus the  ASD connection and the holomorphic curve do not interact.  In this section we will demonstrate how the matching boundary conditions lead to a normal estimate for $e$.  In fact, we have the following result:
\begin{lemma}
For each $(0,t)\in D_R$ we have
$$\frac{1}{2}|\partial_s e|=|\int_\Sigma \langle f_+(0,t),\nabla_s f_{+}(0,t) \rangle+f_-(0,t) \partial_s f_{-}(0,t)| \leq C e^{3/2}(0,t)$$
\end{lemma}
The proof of this lemma will occupy the rest of this section. Since our estimate is local in $(s,t)$, we may assume that our disk $D_R$ is centered at the origin.  The size of the radius is not relevant and we can set it to $R=1$.    We begin by constructing a convenient gauge for $A=\alpha+\phi ds+\psi dt$.  First, we fix $\alpha(0,0)$.  Note that our matching condition implies that $\alpha(0,0)$ is flat and thus, in view of the compactness of $\Mod$, we can choose any such  $\alpha(0,0)$ to have a uniformly bounded $C^4$-norm.  We now construct a particular gauge for $A$ on $D_{1}\times \Sigma$.  For this, let us use $A$ to parallel transport from $(0,0)\times \Sigma$ to $(0,t)\times \Sigma$ for any $t\in (-1,1)$.  Now, extend to $D_1\times \Sigma$ along the $s$-direction.  In these coordinates, $\phi=0$ on $D_1\times \Sigma$ and $\psi=0$ on $(0,t)\times \Sigma$.  Equation $\ref{eq1}$ implies that
\begin{equation} 
\label{eq3}
\partial_s \alpha+*_2(\partial_t \alpha-d_{\alpha}\psi)=0
\end{equation}
and 
\begin{equation} 
*_2F_{\alpha}+\partial_s \psi =0
\end{equation} with $F_{\alpha}=0$ on $(0,t)\times \Sigma$.  This implies that $\partial_s \psi =0$ on $(0,t)\times \Sigma$.  Therefore, $$\partial_s (d_{\alpha}\psi)=d_{\alpha}\partial_s \psi+(\partial_s \alpha) \psi=0$$ on $(0,t)\times \Sigma$.  Now, the energy $(\ref{eq2})$ is given by $$|\partial_t\alpha-d_{\alpha}\psi|^2+|F_{\alpha}|^2$$ for an ASD-connection. Note that since $\nabla_s=\partial_s $ and $F_{\alpha}=0$ on $(0,t)\times \Sigma$ it follows that $$\partial_s \frac{1}{2}\int_\Sigma |f^+|^2d\mu_\Sigma=\langle \partial_s F_{A} , F_A \rangle_\Sigma = 2\langle \partial_t \alpha, \partial_t \partial_s \alpha \rangle_{\Sigma} $$ on $(0,t)\times \Sigma$.  We may apply $\partial_t$ to $(\ref{eq3})$ on $(0,t)\times \Sigma$ to obtain $$*_2\partial_t^2 \alpha + \partial_s \partial_t \alpha=0$$ Therefore, we obtain on $(0,t)\times \Sigma$ $$\partial_s \int_\Sigma |f^+|^2d\mu_\Sigma=-\langle \partial_t \alpha, *_2 \partial_t^2 \alpha \rangle_\Sigma=\int_{\Sigma} tr(\partial_t \alpha \wedge \partial_t^2 \alpha)$$  
We now establish normal estimates on $D_1^-$. We let $\alpha^-(0,0)=\alpha^+(0,0)$.  Along $(0,t)$ we lift $[\alpha^-(0,t)]$ to $\alpha^-(0,t)\in \Aff$ by requiring that $$d^*_{\alpha^-}\partial_t \alpha^-=0$$  
Now, we extend to $D^-$ by lifting along line segments in the $s$-direction with the requirement that $d^*_{\alpha^-}(\partial_s \alpha^-)=0$.  Thus, for sufficiently small neighborhood of $(0,0)$, we have a local lift such that
\begin{equation}
\label{eq4}
 (\partial_s \alpha^--d_{\alpha^-}\phi)+*_2(\partial_t \alpha^- -d_{\alpha^-}\psi)=0
\end{equation}
 The local existence of such a lift follows from existence of ODE as in \cite{DK}, Chapter 6. Since $F_{\alpha^-}=0$ on $D^-$, we have $$d_{\alpha^-}\partial_t \alpha^-=d_{\alpha^-}\partial_s \alpha^-=0$$ on $D^-$.  On $(0,t)$, applying $d_{\alpha^-}$ to $(\ref{eq4})$ we obtain that $d_{\alpha^-}^*d_{\alpha^-} \psi=0$ which implies that $\psi=0$ since $H^0_{\alpha^-}(\Sigma; \liea)=0$. Similarly, applying $d^*_{\alpha^-}$ on $(s,t)\times \Sigma$ we obtain that $\phi=0$ on $(s,t)\times \Sigma$. Therefore, along $(0,t)$ the energy is given by $$e^-=\int_\Sigma |\partial_t \alpha^-|^2$$ while,  $$\partial_s \partial_t\alpha^-+*_2\partial_t^2 \alpha^-=0$$  We have $$-\partial_s e^-= \int_{\Sigma}\langle \partial_t \alpha^-, \partial_t \partial_s \alpha^- -d_{\alpha^-} \partial_s \psi \rangle =  \int_\Sigma \langle \partial_t \alpha^-, \partial_t\partial_s \alpha^- \rangle $$ since $d^*_{\alpha}\partial_t \alpha^-=0$ on $(0,s)$.   We conclude that $$-\partial_s e^-=\int_\Sigma tr(\partial_t \alpha^- \wedge \partial_t^2 \alpha^-)$$ 
For the rest of this section we focus on connections defined on $\Sigma$ parametrized by points in $(0,t)$.  Thus, we will for instance write $\alpha(0,t)$ as $\alpha(t)$.  Take $t\in (-\ep, \ep)\subset (-1,1)$.  We will obtain a $t$-parameter family of flat connections $\tilde{\alpha}(t)$ on $Y_0=[-1,1]\times \Sigma$ with $\tilde{\alpha}(t)_{1 \times \Sigma}=\alpha(t)$ and  $\tilde{\alpha}(t)_{- 1 \times \Sigma}=\alpha^{-}(t)$. \\\\
First, by the matching condition, $\alpha(t)=g^*(t)\alpha^-(t)$.  Taking derivatives, we obtain that 
\begin{equation}
\label{eq5}
\partial_t \alpha= g^{-1}(\partial_t \alpha^-) g + g^{-1}(d_{\alpha^-}\xi) g
\end{equation}

 with $\xi = \partial_t g g^{-1}$.  

Now, for each $t\in (-\ep, \ep)$, we define the extension $\tilde{\xi}(t)$ on $Y_0=[-1,1]\times \Sigma$ with  $\tilde{\xi}_{-1\times \Sigma}(t)=0$, $\tilde{\xi}_{1\times \Sigma}(t)=\xi(t)$  and $$||\tilde{\xi}(t)||_{L^2_{3/2};Y_0}\leq C||{\xi}(t)||_{L^2_1; \Sigma}$$  The existence of such an extension is easy to deduce.  Indeed, we have the surjective restriction map $$\mathfrak{R}:L^2_{3/2}([-1,1]\times \Sigma)\rightarrow L^2_{1}(\partial [-1,1]\times \Sigma )$$  Taking a left inverse to $\mathfrak{R}$ provides such an extension for all $t\in (-\ep, \ep)$.   To extend the gauge transformation $g(t)$ to $Y_0$, we set $\tilde{g}(t)$ to be the unique solution to $\partial_t \tilde{g}(t)=\tilde{\xi}(t)\tilde{g}(t) $ with $\tilde{g}(0)=1$.  \\\\
By pullback, we can regard $\alpha^-(t)$ as a connection on $Y_0$,  Let $\tilde{\alpha}(t)=\tilde{g}^*(t)\alpha^-(t)$ for each $t\in (-\ep, \ep)$. 
By Stokes theorem, 
$$\int_{\partial ([-1,1]\times \Sigma)}tr(\partial_t \tilde{\alpha}\wedge \partial_t^2 \tilde{\alpha})= $$

$$\int_{[-1,1]\times \Sigma}d tr(\partial_t \tilde{\alpha}\wedge \partial_t^2 \tilde{\alpha})= -\int_{[-1,1]\times \Sigma} tr(\partial_t \tilde{\alpha}\wedge 2\partial_t \tilde{\alpha}\wedge \partial_t \tilde{\alpha} )$$

This follows from the fact that $$d_{\tilde {\alpha}} \partial_t \tilde{\alpha}=0$$ since $\tilde{\alpha}$ is flat.  Now, applying $\partial_t$ we obtain $$d_{\tilde \alpha} \partial_t^2 \tilde{\alpha}=-2\partial_t \tilde{\alpha}\wedge \partial_t \tilde{\alpha}$$ as desired.  We claim that $||\partial_t  \tilde{\alpha}(0)||_{L^3;Y_0}\leq C e^{1/2}(0,0)$ for some uniform $C$.  By equation $\ref{eq5}$, $$||\partial_t \tilde{\alpha}(0)||_{L^3;Y_0}\leq C(||\partial_t \alpha^-(0)||_{L^3;\Sigma}+||d_{\alpha^-}\tilde{\xi}(0)||_{L^3;l \Sigma})$$  We have $ ||\partial_t \alpha^-(0)||_{L^3} \leq C e^{1/2}(0,0)$ from the definition and our choice of lift. Now, $$||d_{\alpha^-}\tilde{\xi}(0)||_{L^3}\leq C ||d_{\alpha^-}\tilde{\xi}(0)||_{L^2_{1/2}}\leq C'||\tilde{\xi}(0)||_{L^2_{3/2}}\leq C''||{\xi(0)}||_{L^2_{1}} $$  On the other hand, $||{\xi(0)}||_{L^2_{1}}\leq C e^{1/2}(0,0)$ since we have the bound $$||\partial_t \alpha(0)||_{L^2; \Sigma}\leq e^{1/2}(0)$$ while $H^0_{\alpha^-}(\Sigma; \liea)=0$ and $\alpha^-(0)$ can be taken to vary in a precompact set in say the $C^2$-norm.  Thus, we obtain that $$\partial_t e(0,0)=\int_{\partial [-1,1]\times \Sigma}tr(\partial_t \tilde{\alpha}(0)\wedge \partial_t^2 \tilde{\alpha}(0))\leq C e^{3/2}(0,0)$$  

\subsection{Integration By Parts}
First, we recall some basic Sobolev embedding and restriction results in dimension 2:  

\begin{lemma}
Consider an $L^2_1$ function g on $D_{2R}(p)$, vanishing near the boundary given by $S_{2R}$.  For any $q\in [1,4]$ and large $C'$, we have:

$$||g||_{L^q,D_{R}(p)} \leq C' R^{2/q}||\nabla g||_{L^2, D_{2R}(p)}$$
$$||g||_{L^q, \partial D_{R}(p)} \leq C'R^{1/q} ||\nabla g||_{L^2, D_{2R}(p)}$$
\end{lemma}
\begin{proof}
First, let $R=1$.  Since $g$ is assumed to have compact support on $D_2(p)$, we have by the Sobolev embedding $L^2_1\subset L^q$ (see  \cite{Taylor1}):
$$||g||_{L^q, D_{1}(p)} \leq C' ||\nabla g||_{L^2, D_{2}(p)}$$
$$||g||_{L^q, \partial D_{2}(p)} \leq C' ||\nabla g||_{L^2, D_{2R}(p)}$$ for some $C'>0$.
Now, if we scale $g$ to $g_R$  on a disk of radius $R$, note that  $||\nabla g||_{L^2}$ is scale invariant.  On the other hand $$||g||_{L^q, D_{1}(p)}=R^{-2/q}||g_R||_{L^q, D_{R}(p)}$$ and $$||g||_{L^q, \partial D_{1}(p)}=R^{-1/q}||g_R||_{L^q, \partial D_{R}(p)}$$
\end{proof}
Given a matched pair on some $D_{2R}$ we define the integral energy as 
$$E_{2R}=\int_{D_{2R}}e$$

Here is the key result which allows one to bound the $L^2_1$-norm of $f_{\pm}$ in terms of energy:
\begin{lemma}
There exists $C,C'>0$ with the following property. If $e(s,t) \leq C R^{-2}$ on $D_{2R}$, we have $$||\nabla_Af_+||^2_{D_R }\leq C'R^{-2}E_{2R}$$
$$||\nabla f_-||_{D_R}^2\leq C'R^{-2}E_{2R}$$  
\end{lemma}
\begin{proof}
Let $0\leq \rho \leq 1$ be a radially symmetric bump function supported on $B_{2R}$ such that $\rho =1$ on $B_R$.  We can construct $\rho$ once and for all on $B_2$ and extend to $\rho_R$ for all $B_R$ by dilation.  We have a uniform bound on $||\rho_R||_{C^0}$ and $||\nabla \rho_R||_{C^0} \leq C'R^{-1}$ on $B_{2R}$.   Consider now $f_-$.  We  multiply
$$\nabla^* \nabla f_-=\Delta_2 f_- \leq C'(f_-+f_-^3)$$ on both sides by $\rho^2_Rf_-$ and integrate over $D_{2R}$ to obtain $$\langle \rho^2_Rf _-, \nabla^*  \nabla f_- \rangle_{D_{2R}}\leq C'\int_{D_{2R}}(\rho^2_Rf_-^2+\rho^2_Rf_-^4) $$ Now, we move $\nabla^*$ to the LHS obtaining 
 $$\langle\nabla (\rho^2_Rf _-),   \nabla f_- \rangle_{D_{2R}}\leq C'\int_{D_{2R}}(\rho^2_Rf_-^2+\rho^2_Rf_-^4) -\int_{I_R}\rho^2_R f_- \partial_s f_-$$ 
We move a $\rho_R$ to the RHS:
$$\langle\nabla (\rho^2_Rf _-),   \nabla f_- \rangle_{D_{2R}}-\langle\nabla (\rho_Rf _-),   \nabla (\rho_R f_- )\rangle_{D_{2R}}=-\langle\nabla (\rho_R)f _-,   \nabla (\rho_R) f_- \rangle_{D_{2R}}$$
Thus, we obtain after adjusting $C'$:
$$||\nabla (\rho_R f_-)||^2 \leq C'R^{-2}||f_-||^2_{D_{2R}}+C'\int_{D_{2R}}(\rho^2_Rf_-^2+\rho^2_Rf_-^4) -\int_{I_R}\int_{\Sigma} \rho_R^2 f_- \partial_s f_-$$
Finally, we use the hypothesis that $f_-\leq CR^{-2}$ to obtain the bound 
$$||\nabla (\rho_R f_-)||^2 \leq C'R^{-2}E_{2R}-\int_{I_R}\int_{\Sigma} \rho^2_R f_-  \partial_s f_-$$ after adjusting $C'$ once more.  \\\\
Now, we turn to $f_+$.   We multiply both sides of equation ($\ref{eq2.2.0}$) by $\rho_R^2f_+$ and proceeding as we did with $f_-$ we obtain 
$$||\nabla_A (\rho_R f_+)||^2 \leq C'(R^{-2}||f_+||^2+\int_{D_R\times \Sigma}f_+ (\rho_R f_+)^2)-\int_{I_R}\int_{\Sigma}\langle \rho_R^2 f_+, \nabla_s f_+\rangle $$
$$\leq C'(R^{-2}||f_+||^2+(\int_{D_{2R}} f^2_+)^{1/2}(\int_{D_R\times \Sigma}(\rho_R f_+)^4)^{1/2})-\int_{I_R}\int_{\Sigma}\langle \rho_R^2 f_+, \nabla_s f_+\rangle $$ 
We  use the Sobolev embedding $$L^2_1 \subset L^4$$ in dimension $4$ together with Kato's inequality 
\begin{equation}
\label{Kato}
\nabla |f|\leq |\nabla_A f|
\end{equation}
 to bound $$(\int_{D_{2R}} f^2_+)^{1/2}(\int_{D_R\times \Sigma}(\rho_R  f_+)^4)^{1/2}) \leq C'(\int_{D_{2R}} f^2_+)^{1/2}(||\nabla_A (\rho_R f_+)||^2+||\rho_R f_+||^2)$$ 
By assumption, $\int_{D_{2R}}f_+^2 \leq C$ and thus if $C<1/C'$ we can absorb the term to the LHS to obtain the bound $$||\nabla_A (\rho_R f_+)||^2 \leq C'R^{-2}||f_+||^2_{D_{2R}}-\int_{I_R}\int_{\Sigma}\langle \rho_R^2 f_+,\nabla_s f_+\rangle=C'R^{-2}E_{2R}-\int_{I_R}\int_{\Sigma}\langle \rho_R^2 f_+,\nabla_s f_+\rangle$$ after adjusting $C'$.
Since, by section $\ref{secnormal}$, we have $$|\int_{\Sigma}\langle f_+,\nabla_s f_+\rangle ++\int_{\Sigma}f_- \partial_s f_- | \leq C' e^{3/2}\leq C'C^{1/2} R^{-1}e$$
we can apply the Sobolev embedding lemmas in dimension 2 to conclude that $$\rho^2_R |\int_{(0,t)\in D_R}\int_{\Sigma}\langle f_+,\nabla_s f_+\rangle+\int_{\Sigma}f_- \partial_s f_- | \leq C^{1/2}C''(||\nabla_A(\rho_R f_+)||^2+||\nabla(\rho_R f_-)||^2)$$  Taking $C$ sufficiently small, we can absorb the term  $C^{1/2}C''(||\nabla_A(\rho_R f_+)||^2+||\nabla(\rho_R f_-)||^2)$ to the left hand side and obtain the desired inequality.
\end{proof}

\subsection{Inverting the Laplacian}
Let us summarize the situation so far.  We have a nonnegative function $$e:D_{R_0}\subset H^+ \rightarrow [0,\infty) $$ with energy functional $$E_R=\int_{D_R}e$$  The function $e$ satisfies the following properties:\\

There exists $C>0$,$C'>0$ such that if $e\leq CR^{-2}$ on $D_{2R}$ we have:
\begin{enumerate}
	\item $(\int_{D_R} e^2)^{1/2} \leq C' E_{2R}R^{-1} $
	\item $\int_{\partial D_R} e \leq C'E_{2R}R^{-1}$
	\item $(\int_{\partial D_R} e^2)^{1/2} \leq C' E_{2R}R^{-3/2}$
	\item $\Delta_2 e \leq C'(e+e^2 + e^{1/2}v) $ with $(\int_{D_R}v^2 )^{1/2}\leq C' E_{2R}R^{-2}$
 \item $| \partial_s e(0,t)| \leq C'e^{3/2}(0,t)$
\end{enumerate}
Our immediate objective is to use these assumptions to get a pointwise bound on $e$ in terms of energy:
\begin{lemma}
For some $C''>0$, and each $D_R$ in the domain, we have $$e(p)\leq C''E_{R}R^{-2}$$ where $p$ is the center of $D_R$
\end{lemma}
The proof of this lemma occupies the rest of this section.  First, given $C^2$ functions $f$, $g$, on $D_R$, Green's formula in dimension 2 gives
$$\int_{D_R} f \Delta_2 g -\int_{D_R} g \Delta_2 f =\int_{\partial D_R} (f\partial_\nu g -g\partial_\nu f )   $$
  In general, $\partial D_R=I_R\cup S_R$ where $$I_R=\{(0,t)\in D_R \}$$ and $$S_R=\{(s,t)\in \C| s^2+t^2=R^2, s\geq 0 \}$$
If $g=\ln(R)-\frac{1}{2}\ln((s-s_0)^2+t^2)$, which is a Green's function for $\Delta_2$ at $p=(s_0,0)$, we get 
\begin{equation}
\int_{\partial D_R} (f\partial_\nu g -g\partial_\nu f ) = \int_{I_R}(\frac{s_0f}{s_0^2+t^2}+(\ln(R)-\frac{1}{2}\ln(s_0^2+t^2))\partial_s f  )-R^{-1}\int_{S_R}f 
\end{equation}
\textbf{Case 1:} $D_R$ is the whole disk of radius $R$.\\\\
In this case $\partial D_R=S_R$.  For concreteness we may assume that the disk is centered at $p=0$ (one must shift $H$ as well but $H$ does not interact with $D_R$ in this case).   We multiply both sides of (4) by $\ln(R)-\ln(r)$ and use Green's formula to obtain $$2\pi e(p)  \leq C' (\int_{D_R} e (\ln(R)-\ln(r)) + \int_{D_R} e^2 (\ln(R)-\ln(r))+ \int_{D_R} e^{1/2} v (\ln(R)-\ln(r))+R^{-1}\int_{\partial D_R} e) $$
We now bound each of the four terms. For this, it will be convenient to recall the following definite integrals:
$$\int_0^R (\ln(r)-\ln(R))^2dr=2R$$ 
$$\int_0^R (\ln(r)-\ln(R))^2r dr=R^2/4$$
We can now use the property (1)-(4) of $e$ to bound the terms as follows:
 $$\int_{D_R} e (\ln(R)-\ln(r)) \leq (\int_{D_R} e^2)^{1/2} (2\pi \int_0^R (\ln(R)-\ln(r))^2rdr)^{1/2}\leq 2\pi  C'E_{2R}$$
 $$\int_{D_R} e^2 (\ln(R)-\ln(r)) \leq C R^{-2} (\int_{D_R} e^2)^{1/2} (2\pi \int_0^R (\ln(R)-\ln(r))^2rdr)^{1/2}\leq 2\pi C'R^{-2}E_{2R}$$
 $$\int_{D_R} e^{1/2} v (\ln(R)-\ln(r)) \leq (C^{1/2} R^{-1})(\int_{D_R} v^2)^{1/2} (2\pi \int_0^R (\ln(R)-\ln(r))^2rdr)^{1/2}\leq 2\pi C'R^{-2}E_{2R}$$
 $$R^{-1}\int_{\partial D_R}e \leq C' R^{-2}E_{2R} $$
 Here we bound $e^{1/2}$ using the fact that $e \leq CR^{-2}$ on $D_{2R}$.  Putting this together yields $$e(p) \leq C''E_{2R}R^{-2}$$ as desired.\\\\
 \textbf{Case 2:} $D_R$ is the half disk with center $p=0$.  In this case we have an extra contribution to Green's formula given by $$\int_{-R}^R (\ln(r)-\ln(R))\partial_s edr$$  Since property (4) of $e$ implies that $| \partial_s e(0,t)| \leq C' e^{3/2}(0,t)$ we get $$|\int_{-R}^R \partial_t e (\ln(r)-\ln(R))dr| \leq C' \int_{-R}^R e^{3/2} (\ln(r)-\ln(R))dr$$ $$ \leq C^{1/2}R^{-1} (\int_{-R}^R e^2dr)^{1/2}(\int_{-R}^R (\ln(r)-\ln(R))^2dr)^{1/2} \leq C'' E_{2R}R^{-2}$$
\textbf{Case 3:} $s_0\leq R/2$.  In this case there are two extra boundary terms in Green's formula $$\int_{I_R}(\frac{s_0e}{s_0^2+t^2}+(\ln(R)-\frac{1}{2}\ln(s_0^2+t^2))\partial_s e  )$$  However,  $e(0,t)$ lies on the boundary and is contained in a disk of radius $R/2$ centered at $0$.  Therefore, we may apply the previous case to conclude that $e(0,t) \leq C'' E_{2R}$ for $C''$ sufficiently large and independent of $R$ and $s_0$.   We have the bound  $$\int_{I_R}\frac{s_0^2}{s_0^2+t^2}dt\leq \int_\Ar \frac{1}{1+t^2}dt <\infty $$ which takes care of this term.  To bound the second term we must bound
$$\int_{I_R}(ln(R)-\frac{1}{2}\ln(s_0^2+t^2))^{2}dt$$ as in Case 2.  Rescaling $(s,t)$ to $(s/R,t/R)$ and note that by continuity and the fact $0\leq s_0 \leq 1/2$, $$\int_{-\sqrt{1-s_0^2}}^{\sqrt{1-s_0^2}} \ln(s_0^2+t^2)^2 dt<C''$$ for $C''$ is large.  This implies that 
$$\int_{I_R}(ln(R)-\frac{1}{2}\ln(s_0^2+t^2))^{2}dt \leq C''R$$ as desired.\\\\
\textbf{Case 4:}  $s_0 \geq R/2$.  In this case the whole disk of radius $R/2$ does not intersect $H$.  We may apply Case 1 to conclude that $e(p)\leq C''R^{-2}E_{2R}$, after adjusting $C''$. \\\\

 Thus, we have demonstrated that if $e(z)\leq C R^{-2}$ on $D_{2R}$, we have $e(p) \leq C''R^{-2}E_{2R}$. Once we replace $R$ by $R/2$ and rescale the constants by $4$, we may finally conclude that if $$e(z) \leq CR^{-2}$$  on $D_R$ we have $$e(p) \leq C''R^{-2}E_{R}$$ as desired.

\subsection{Completing the Proof -  A Continuity Argument}
We are now in position to remove the pointwise assumption on $e$  and replace it by an integral energy assumption.  This will allow us to complete the proof of theorem $\ref{thm2.1}$.  For convenience, let us abstract the relevant setup.  For a given $R_0>0$, let  $$e:D_{R_0}(p_0)\subset H^+\rightarrow [0,\infty)$$ be a continuous function and let $$E_{D_R(p)}=\int_{D_R(p)}e$$ for all $D_R(p) \subset D_{R_0}(p_0)$. Suppose,  $$e(p)\leq C''E_{D_R(p)}/R^2$$ whenever $e \leq C/R^2$ on $D_{R}(p)$.   
\begin{theorem}
Let   $\hbar =C/17C''$. For any $D_R(p)\subset D_{R_0}(p_0)$ with $E_{D_R(p)}\leq \hbar$, we have $$e(p) \leq 4CC'' E_{D_R(p)} /R^2 $$   
\end{theorem}\begin{proof}
The proof of this theorem is variant of a continuity argument which we reproduce for completeness.  For notational simplicity we will shift $H^+$ so that $p=0$ and rescale $e$ so that $C=1$. On $D_{R}(0)$, let $$\rho(r)=(R-r)^2 \sup_{z \in D_{r}(0)}e(z) $$  for $r\in [0,R]$.  Since $\rho$ is a continuous function on a compact domain it must have a maximum at some $r_0>0$.  Let $z_0\in D_{r_0}$ be a point where $\rho(r_0)=(R-r_0)^2e(z_0)$. Note that $r_0=|z_0|$.  If $\rho(r_0)\leq 1/4$, then on $D_{R/2}(0)$, $$\sup_{z\in D_{R/2}}e(z) \leq 1/R^2$$ and from our hypothesis we can deduce that $$e(0) \leq C''E_{D_{R/2}(0)}/(R/2)^2\leq 4C''E_{D_{R}(0)}/R^2 $$ as desired.  Thus, we may assume $$\rho(r_0)=(R-|z_0|)^2 e(z_0) >1/4$$  Let $s^2=\frac{1}{16e(z_0)}$.  We have $$s< (R-|z_0|)/2$$  Therefore, $D_s(z_0) \subset D_R(0)$. \\\\ Since $\rho(|y|) \leq \rho(|z_0|)$ for any $y \in D_s(0)$, we have $$e(y) \leq \frac{e(z_0) (R-|z_0|)^2}{(R-|y|)^2} \leq  \frac{e(z_0) (R-|z_0|)^2}{(R-|z_0|-R/2+|z_0|/2)^2} = 4e(z_0)=\frac{1}{4s^2}$$
  Therefore, for $y\in D_{s}(z_0)$ the hypothesis $$e(y)\leq \frac{1}{s^2}$$ is valid and implies  $$e(z_0) \leq C''E_{D_s(z_0)}/s^2=16C''e(z_0)E_{D_s(z_0)} \leq 16C''e(z_0)E_{D_R(0)} < e(z_0)$$ since by hypothesis  $E_{D_R(0)}\leq 1/17C''$.  This is a contradiction and thus $\rho(z_0) \leq 1/4$.    
\end{proof}

Note that $\hbar$ above is independent of $R_0$.  Therefore, we have completed the proof of theorem $\ref{thm2.1}$.     

\section{Regularity and Convergence}
\label{secreg}
\subsection{Interior Regularity for the ASD equation}
The goal of the present section is to establish regularity results for the matched equations.  We will begin with a review of the proof of interior regularity for the ASD equation.  This material is rather standard and covered in many sources (see \cite{DK}).  We have chosen to include a brief discussion to facilitate the treatment of the rather involved regularity result for the matched equations.   Let $X$ be a smooth Riemannian 4-manifold and let $A_0$ be some fixed smooth $SO(3)$-connection.  Let $A$ be an ASD connection:  $$F_{A}^{+}=0$$ As a stationary point of the Yang-Mills functional, $A$ is automatically a Yang-Mills connection $$d_{A}^*F_{A}=0$$
Recall that $A$ is in \textbf{Coulomb gauge} with respect to $A_0$ if $$d_{A_0}^*(A-A_0)=0$$ 
Assume that $A$ is ASD and in Coulomb gauge with respect to some fixed smooth $A_0$. We will tailor the dicussion to the case when $X=D_R\times \Sigma$ although the results have direct generalization to any $X$.  Since in this section we are dealing with the interior case, we assume that $D_R$ does not intersect the boundary. \\\\ 
Fix some $p>2$.  Let $p_0=1/p$, $q_1=2p_0-1/4$ and $p_1=2p_0-1/2$. Note the embedding $$L^{1/q_1}_2\subset L^{1/p_1}_1$$
Here is the main result we will need:
\begin{lemma}
If  $A$ is $L^{1/p_0}_1$, then $A$ is in $L^{1/q_1}_2$.  If $A\in L^p_k$ for $k>1$, then $A\in L^p_{k+1}$. \\ Take any $R'<R$.  Suppose we are given a sequence of ASD connections $A_i$ on $X$ in Coulomb gauge with respect to some fixed $A_0$.   If  $A_i$ converges in $L^{1/p_0}_1$ on $D_R\times \Sigma$ then $A_i$ converges in $L^{1/q_1}_2$ on any $D_{R'}\times \Sigma$.   For any $k>1$, suppose $A_i$ converges in $L^p_k$ on $D_R\times \Sigma$.  Then, $A_i$ converges in $L^p_{k+1}$ on any $D_{R'}\times \Sigma$. 
\end{lemma}
Using this lemma, one may immediately deduce regularity and convergence properties of a sequences of ASD connections on $X$.  We will discuss this in more detail at the end of the section.  The proof of this lemma occupies the rest of this subsection. Let $A=A_0+B$ be an ASD connection in Coulomb gauge with respect to $A_0$.  Thus 
\begin{equation}
d_{A_0}^* B=0
\end{equation}
\begin{equation}
F_{A_0+B}=F_{A_0}+d_{A_0}B+B\wedge B
\end{equation}
If we project to the self-dual part of the curvature, we obtain 
\begin{equation}
(d_{A_0}^++d_{A_0}^*)B=-F_{A_0}-(B\wedge B)^+
\end{equation}
Now, applying 
\begin{equation}
2d_{A_0}^*+d_{A_0}:\Omega^+(X; \liea)\oplus \Omega^0(X; \liea)\rightarrow \Omega^1(X; \liea)
\end{equation}
 we obtain that 
\begin{equation}
\Delta_0 B=B'\cdot B+g_0
\end{equation} 
Where $\Delta_0=d_{A_0}^*d_{A_0}+d_{A_0}d_{A_0}^*$ is the Hodge Laplacian, $g_0$ is a smooth function that depends only on $A_0$, $$B\mapsto B'$$ is the action of a smooth first order operator acting on $B$ and $B'\cdot B$ is some algebraic multiplication.  Since we will be concerned with estimates on $D_{R'}\times \Sigma$, let $$\rho: D_R\rightarrow \Ar$$ be a bump function with support in $D_R$ such that $\rho=1$ on $D_{R'}$.  We have 
\begin{equation}
\Delta_0 (\rho B)=B'\cdot (\rho B)+g_0+L(B)
\end{equation}
where $L$ is some first order differential operator that depends on $\rho$.  
 In the case when $k=1$, this equation must be interpreted in the weak sense.  In other words, given any smooth section $s$ of $\Omega^0(X; \liea)$ with support in $D_{R}\times \Sigma$, we have $$\langle \rho B,
\Delta_0 s  \rangle =\langle B'\cdot (\rho B)+g_0+L(\rho B), s \rangle $$  

\noindent First, we obtain $L_{2}^{1/q_1}$-regularity for $\rho B$.  For this, we use the embedding in equation $(\ref{eq13})$ to obtain an $L^{1/q_1}$ bound on $B'\cdot (\rho B)$.  Now, we apply the regularity results of theorem $\ref{thm4}$ to obtain $L_{2}^{1/q_1}$-bounds on $\rho B$.  \\\\
Now, we assume that we are in the stable range $pk>4$ and $k>1$.    The embedding  $$L^p_k\cdot L^p_{k-1}\rightarrow L^p_{k-1}$$ implies an $L^p_{k-1}$ bound on $B'\cdot \rho B$ and hence elliptic estimates give an $L^p_{k+1}$-bound on  $\rho B$. \\\\ The sequential version of this argument follows a similar pattern.  First,  $B_i-B_j$ satisfies    

\begin{equation}
\begin{split}
\Delta_0 (\rho (B_i-B_j))&=B_i'\cdot (\rho B_i)-B_j'\cdot (\rho B_j)+L(B_i)-L(B_j)\\
&=B_i'\cdot (\rho (B_i-B_j))-(B_j'-B_i')\cdot (\rho B_j)+L(B_i)-L(B_j)
\end{split}
\end{equation}

Now, arguing as above using the Sobolev embeddings, we may conclude that $A_i$ converges on $D_{R'}\times \Sigma$ as desired.  

\subsection{Regularity For J-Curves in a Banach Space}
\label{JRegSec}
We now turn to the discussion of regularity for holomorphic curves with values in a Banach space.  Let $$B^0_p=L^p(\Sigma)\times \Ar^N$$ and let $B_p=B^0_p\oplus B^0_p$.  We will assume that $B_p$ has a smooth almost complex structure $$J:B_p\rightarrow End(B_p,B_p)$$  Furthermore, we assume that, along $\Lag=0\oplus B_p^0$, $J$ is given by $J_0$ where $$J_0(b_0,b_1)=(-b_1, b_0)$$  Thus, $\Lag$ is totally real with respect to $J$.  \\\\
In this section, take $D_R$ the be centered at the origin and let $p>2$, $\geq 1$ be a map
\begin{equation}
u:D_R\rightarrow B_p 
\end{equation}
such that $u\in L^p_k(D_R; B_p)$.  Furthermore, assume that $u_{|\partial D_R}$ maps to $\Lag$.  Such a map $u$ is said to be $J$-holomorphic if $$\delbar_J(u)=\partial_tu+J(u)\partial_s u =0$$  Here is the basic technical result we will need:
\begin{lemma}
\label{lem19}
Let $k > 1$.  
Given $u\in L^p_k(D_R; B_p)$ with $\delbar_Ju=g\in L^p_k(D_R; B_p)$ we have $u\in L^p_{k+1}(D_R; B_p)$.  Given a sequence $u_i\in L^p_k$ converging on ${L^p_k(D_R)}$  we have that $u_i$ converges in  ${L^p_{k+1}(D_{R'}; B_p)}$ for any $R'\subset R$.
\end{lemma}
\begin{proof}
Our  strategy is to reduce the problem to a regularity result for the Laplacian.  Write $u=(u_0,u_1)$ using the decomposition $B_p=B_p^0\oplus B^0_p$.  By assumption, $u_0$ vanishes on $\partial D_R$.  On the other hand, $u_1$ satisfies the normal boundary condition $\partial_t u_1=g_{|\partial D_R}$.  Applying $\partial_t-J(u)\partial_s$ to $$\delbar_Ju=g$$ we obtain $$(\partial_s^2+\partial_t^2)u=-\Delta u=-J(u)g'+g'+(J'(u)(u'))\cdot u'$$ Here, as well as in the sequel, we will use $u'$ to denote some first order differential operator on $u$ with smooth coefficients.  By assumption, $J(u)g'\in L^p_{k-1}$ and $$(J'(u)(u'))\cdot u' \in L^p_{k-1}$$ in view of the product theorem $$L^p_l\cdot L^p_l\rightarrow L^p_l$$ as long as $p>2$ and $l\geq 1$.  Elliptic estimates from section $\ref{BanahEst}$, imply that $$u \in L^p_{k+1}(D_R; B_p)$$ as desired.  Consider now the sequential version.  For any given point $p\in D_R$ We will produce a uniform bound on a neighborhood $D_r(p)$ of $p$.  By assumption, $u_i(p)$ converge in $B_p$.  We may therefore, take $r$ sufficiently small that $J$ is uniformly bounded in $C^{k+3}$ on the image of each $(u_i)_{|D_{2r}(p)}$.  As above, we have $$\Delta u_i=-J(u_i)g'_i+g_i'+(J'(u_i)(u'_i))\cdot u'_i$$ In view of the uniform bound on $J(u_i)$ in $C^{k+3}$, we have uniform bounds on the $L^p_{k-1}$-norm of $(J'(u_i)(u'_i))\cdot u'_i$.  We now apply the regularity estimates to obtain a uniform bound on $D_{R'}$. 
\end{proof}
We now address the case when $k=1$.  We have:
\begin{lemma}
\label{Jreg2}
Take $p>2$.  Given $u\in L^p_1(D_R; B_p)$ with $\delbar_Ju=g\in L^p_1(D_R; B_p)$ then $u\in L^{p/2}_{2}(D_R; B_{p/2})$.  Given a sequence $u_i\in L^p_1$ converging on ${L^p_1(D_R)}$,  then $u_i$ converges on  ${L^{p/2}_{2}(D_{R'}; B_{p/2})}$. 
\end{lemma}
\begin{proof}
We imitate the proof of the result above.  However, this time the product theorem maps
$$L^p(D_R; B_p)\times L^p(D_R; B_p)\ \rightarrow L^{p/2}(D_R; B_p) \rightarrow  L^{p/2}(D_R; B_{p/2}) $$
The last embedding is necessary since we have not developed regularity theory for mixed spaces such as $L^p(D_R;B_q)$ where $p\neq q$.  
\end{proof}
\textbf{Remark.} The lemma above will be most useful to us when $p>4$.  To address the case when $4\geq p>2$ we utilize a specialized argument:
\begin{lemma}
Fix some $p'>p$.  Given $u\in L^p_1(D_R; B_p)\cap  L^{p'}(D_R; B_{p'})$ with $\delbar_Ju=g\in L^{p'}(D_R; B_{p'})$ then $u\in L^{p'}_{1}(D_R; B_{p'})$. Given a sequence $u_i\in L^p_1(D_R; B_p)\cap  L^{p'}(D_R; B_{p'})$ converging in $ L^p_1(D_R; B_p)\cap  L^{p'}(D_R; B_{p'})$  then $u_i$ converges in  ${L^{p'}_{2}(D_{R'}; B_{p'})}$.   
\end{lemma}
\begin{proof}
We will prove regularity around an arbitrary point $x$ in $\partial(D_R)$.  For convenience, we take $x=0$.  Let $T(s,t)=J_0-J(u(s,t))$.  By construction $T(0,0)=0$. Let $\rho$ be a bump function with support in $D_{R}$ and $\rho=1$ on some $D_{R'}$. Let $v=\rho u$.   We have that $$\delbar_J (\rho u)\in L^{p'}$$ with support in $D_R$.  Since $v$ has support away from $S_R$ of $D_R$, we may view $v$ as a function on the closed disk $D_R$ with Lagrangian boundary conditions.  For this one needs to "round" the corners of $D_R$ but since the support of $v$ vanishes around there it does not affect the argument.    On $D_R$, $v$ satisfies $$\partial_{J_0}v + T(s,t)\partial_t v =0$$  Assuming that the support of $\rho$ is sufficiently small, the norm of $$T:L^{p'}(D_R;B_{p'})\rightarrow L^{p'}(D_R;B_{p'})$$ as well as $$T:L^{p}(D_R;B_{p})\rightarrow L^{p}(D_R;B_{p})$$ is small.  It follows that the operator $\partial_{J_0} + T(s,t)\partial_t $ is surjective as an operator $L^p_1\rightarrow L^p$ as well as $L^{p'}_1\rightarrow L^{p'}$ and the kernel consists of constant solutions on the Lagrangian $\Lag$.  This implies that $v\in L^{p'}_1$ as desired.      The convergence argument is similar. 
\end{proof}

\subsection{Regularity for the Matched Equations}
\label{regsec}
We now turn to the regularity results for the case of matched boundary conditions. Let $D_R\subset \C$ be a disk centered at the origin.  
For some $p>2$, consider a $J$-curve  $$u:D^-_R\rightarrow \Mod$$ with $u \in L^p_1(D^-_R)$.
and an ASD connection $A\in L^p_1(D_R\times \Sigma)$ on $D_R\times \Sigma$.  We decompose $A$ as $$A=\alpha+\phi ds+\psi dt$$ with $\alpha(s,t) \in \Aff$.  We will assume that $(u,A)$ are matched at the boundary $$u(s,0)=[\alpha(s,0)]$$
Here $[\alpha(s,t)]$ denotes the equivalence class in $\Mod$.   By the Coloumb slice theorem (see \cite{Kat1}), there exists a smooth connection $A_0$ on $D_R\times \Sigma$  and a gauge transformation $g\in L^p_{2}$  with the following properties.  If $B=A-A_0$, then, $$d_{A_0}^*(B)=0$$ and $$B(s,0)(\partial_t)=0$$ Here is our main regularity theorem:
\begin{theorem}
\label{thmreg}
Let $(u,A)$ be a matched pair on $D_R\times \Sigma$.  Assume $u\in L^p_1$ and $A\in L^p_1$ for $p>2$.  Furthermore, assume that $u$ is $J$-holomorphic and $A$ is ASD.  If $A$ is in Coulomb gauge with respect to a smooth connection $A_0$, then $(u,A)$ is smooth.  
\end{theorem}

The proof this result occupies the rest of this section.  Our strategy is to prove regularity for the different components of $A$ separately.  A variant of this strategy with for a different boundary value problem appears in \cite{Kat2}.

\subsubsection{Estimates on $\psi$}
Let us decompose $A_0$ as $$A_0=\alpha_0+\psi_0 +\phi_0$$  Since $\Delta_0B+B\cdot B' =0$, we may project to the $dt$-component to deduce that $$\Delta_0(\psi-\psi_0)+L(B\cdot B')=0$$
 where $L(B\cdot B')$ is the projection of $B\cdot B'$ to the $dt$-component. By assumption, $\psi(0,s)=\psi_0(0,s)$.  Since $\psi_0$ is smooth we are in position to apply the Dirichlet boundary value problem estimates to deduce regularity of $\psi$.  We summarize the bootstrapping estimates with the lemma below:

\begin{lemma}
\label{lemreg1}
 Let $p_0=1/p$, $q_1=2p_0-1/4$ and $p_1=2p_0-1/2$. If $A\in L^p_1(D_R\times \Sigma)$, then $$||\psi||_{L^{1/p_1}_1(D_{R'}\times \Sigma)}\leq C_{p,1}||A||_{L^{1/p_0}_1(D_{R}\times \Sigma)}$$ and  
 $$||\psi||_{L^{1/q_1}_2(D_{R'}\times \Sigma)}\leq C_{p,1}||A||_{L^{1/p_0}_1(D_{R}\times \Sigma)}$$ 
 If $A\in L^p_k(D_R \times \Sigma)$ for $k>1$, then $$||\psi||_{L^{p}_{k+1}(D_{R'}\times \Sigma)}\leq C_{p,k}||A||_{L^{p}_{k}(D_{R}\times \Sigma)}$$ 
 Furthermore, the constants $C_{p,k}$ do not depend on the choice of $A$. 
 \end{lemma} 
\begin{proof}
The proof of this proposition is identical to that of the interior case aside from the fact that now we base our linear elliptic estimates on the Dirichlet problem which we discussed in section $\ref{SecReg}$.  The nonlinear estimates on $B'\cdot B$ is identical to the one for the interior case and gives rise to the same estimates.  
\end{proof}

\subsubsection{Estimates on $\phi$}
We now address the regularity of $\phi$.  This case is a bit more subtle in view of the boundary condition.  Since $A$ is ASD, we have $$*_2 F_\alpha+\partial_s \psi+\partial_t \phi+[\psi,\phi]=0$$  By assumption, $\psi$ is smooth at the boundary. In addition, since $(u,A)$ are matched, $*_2 F_\alpha$ vanishes at the boundary.  The basic strategy now is to apply the elliptic theory for the Neumann problem to obtain regularity for $\phi$.  However, we must be careful since we are initially starting with an $L^p_1$-configuration and we must discuss a weak version of the Neumann boundary value problem.  Let us first introduce some notation.  Let $I_\tau$ the set  of points $(s,\tau)\in D_R$.  Thus, $I_0\times \Sigma$ is the matched boundary of $D_R\times \Sigma$.  Let $f$ be a smooth function with support in $D_R\times \Sigma$ such that $\partial_t f=0 $ on $I_0 \times \Sigma$.  We now check that $\phi$ satisfies a weak version of the Neumann problem:  
\begin{lemma}
$$\langle \phi , \Delta_{0} f \rangle_{D_R\times \Sigma}  =  \langle \Delta_{0} \phi , f \rangle_{D_R\times \Sigma} +\langle [\phi, \psi], f \rangle_{I_0\times \Sigma}$$ 
where $\Delta_{0} \phi $ is the linear projection onto the $ds$-component of $B\cdot B'$.  
\end{lemma}
\begin{proof}
Let $$D_r \{ (s,t) \in D_R| t \geq r \}$$  Since $\phi$ is smooth on $D_r$ for $r>0$, we obtain 
$$\langle \phi , \Delta_0 f \rangle_{D_r\times \Sigma} =\langle \Delta_0 \phi , f \rangle_{D_r\times \Sigma} + \langle \phi , \partial_t f \rangle_{I_\tau \times \Sigma} + \langle *_2 F_{\alpha}^0  ,  f \rangle_{I_\tau \times \Sigma}+ \langle \ [\phi,\psi] , f \rangle_{I_\tau \times \Sigma} $$
\end{proof}

We need to argue that the last 3 terms approach $0$ as $ \tau \rightarrow 0$.  For $\langle \phi , \partial_t f \rangle_{I_\tau \times \Sigma} + \langle \ [\phi,\psi] , f \rangle_{I_\tau \times \Sigma}$ this is straightforward since $\partial_t f=0$ on $I_0\times \Sigma$ and $\phi=0$ on   $I_0\times \Sigma$.  We need to examine the term $\langle *_2 F_{\alpha}  ,  f \rangle_{I_\tau \times \Sigma}$.  Since $A\in L^p_1$, we have $$\phi \in L^p_1(D; L^p(\Sigma))\subset C^0(D; L^p(\Sigma))$$    The moment map sending $\alpha$ to $F_\alpha$ is continuous as a map $L^p(\Sigma)\rightarrow L^p_{-1}(\Sigma)$.  Therefore, $$*_2F_\alpha \in C^0(D, L^p_{-1}(\Sigma))$$  By assumption, $F_\alpha^0=0$ on $I_0\times \Sigma$.  \\
Since $f$ is smooth, it gives a well defined element of $C^0(D; L^p_1(\Sigma))$.  And $$\langle *_2 F_{\alpha}  ,  f \rangle_{I_\tau \times \Sigma} \leq C\cdot \max_{I_{\tau}}|F_\alpha|_{L^p_{-1}}$$ where $C$ depends only of $f$.  Therefore, $$\langle *_2 F_{\alpha}  ,  f \rangle_{I_\tau \times \Sigma}\rightarrow 0 $$  as $\tau \rightarrow 0$ as desired.

\begin{lemma}
\label{lemreg2}
 Let $p_0=1/p$, $q_1=2p_0-1/4$ and $p_1=2p_0-1/2$. If $A\in L^p_1(D_R\times \Sigma)$, then $$||\phi||_{L^{1/p_1}_1(D_{R'}\times \Sigma)}\leq C_{p,1}||A||_{L^{1/p_0}_1(D_{R}\times \Sigma)}$$ and  
 $$||\phi||_{L^{1/q_1}_2(D_{R'}\times \Sigma)}\leq C_{p,1}||A||_{L^{1/p_0}_1(D_{R}\times \Sigma)}$$ 
 If $A\in L^p_k(D_R\times \Sigma)$ for $k>1$, then $$||\phi||_{L^{p}_{k+1}(D_{R'}\times \Sigma)}\leq C_{p,k}||A||_{L^{p}_{k}(D_{R}\times \Sigma)}$$ 
 Furthermore, the constants $C_{p,k}$ do not depend on the choice of $A$. 
 \end{lemma} 
 
 \begin{proof}
This time our estimates are based on the solution to the Neumann problem that we described in section $\ref{Dirichlet}$.  Given $\phi$, we assume that $\phi$ satisfies $$\langle \phi, \Delta_0 f  \rangle_{D\times \Sigma}= \langle g , f \rangle_{D\times \Sigma} +\langle \rho , f \rangle_{I_0\times \Sigma}$$ for all smooth $f$ with $\partial_tf=0$ on $I_0\times \Sigma$ and support in $D\times \Sigma$.  One obtains estimates on $\phi$ from the regularity on $g$ and $\rho$.  On our situation, $\rho=[\psi, \phi]_{I_0\times \Sigma}$ and $g=B'\cdot B$.  In view of the regularity results we obtained on $\psi$, the estimates on $\phi$ from combining more details.
\end{proof}

\subsubsection{Slicewise Estimates on $\alpha$}
The starting observation is that the ASD equation together with the Coulomb gauge condition imply that $$(d_0+d_0^*)(\alpha-\alpha_0)=B\cdot B+\phi'+\psi'+g$$
where $g$ is some fixed smooth function.  In other words, $(d_0+d_0^*)(\alpha-\alpha_0)$ does not involve any $(s,t)$-derivatives of $\alpha$.    We have the following estimates:

\begin{lemma}
\label{lemreg3}
If $A\in L^p_1(D_R\times \Sigma)$ then $$||\alpha||_{L^{1/p_1}_1(\Sigma; L^{1/p_1}(D_R))} \leq C_{p,1}||A||_{L^p_1(D_R\times \Sigma)} $$ and 
 $$||\alpha||_{L^{1/q_1}_2(\Sigma; L^{1/q_1}(D_R))} \leq C_{p,1}||A||_{L^p_1(D_R\times \Sigma)} $$
If $A\in L^p_k(D_R\times \Sigma)$ for $k>1$ then $$||\alpha||_{L^{p}_{k+1}(\Sigma; L^p(D_{R'}))} \leq C_{p,k}||A||_{L^p_k(D_R\times \Sigma)} $$ 
Furthermore, the constants do not depend on $A$.  
\end{lemma}

\begin{proof}
The proof of this proposition combines slicewise elliptic regularity with estimates on the nonlinear terms.  To begin, if $B\in L^p_1$, then by equations $(\ref{eqemb1})$ and ($\ref{eqemb2}$)  $$B\cdot B \in L^{1/p_1}(D_R\times \Sigma)\cap L^{1/q_1}_1(D_R\times \Sigma)$$  We have already established in lemma $\ref{lemreg1}$ and $\ref{lemreg2}$ that $$\psi , \phi \in L^{1/p_1}(D_R\times \Sigma)\cap L^{1/q_1}_1(D_R\times \Sigma)$$  Thus, we have $$(d_0+d_0^*)(\alpha-\alpha_0)\in L^{1/p_1}(D_R\times \Sigma)\cap L^{1/q_1}_1(D_R\times \Sigma) \subset   L^{1/p_1}(\Sigma; L^{1/p_1}(D_R))\cap L^{1/q_1}_1(\Sigma; L^{1/q_1}(D_R))$$ 
We now apply elliptic regularity for $d_0+d_0^* $ with values in the Banach space $ L^{1/p_1}(D)$ (or $ L^{1/q_1}(D_R)$) to deduce that $$\alpha-\alpha_0 \in   L^{1/p_1}_1(\Sigma; L^{1/p_1}(D_R))\cap L^{1/q_1}_2(\Sigma; L^{1/q_1}(D_+))$$
To obtain the higher estimates, one proceeds in a similar fashion. We have $B\cdot B \in L^p_k(D_R\times \Sigma)$ and $\phi, \psi \in L^p_{k+1}(D_R\times \Sigma)$.  Therefore,  $$(d_0+d_0^*)(\alpha-\alpha_0)\in  L^p_k(D_R\times \Sigma)\subset L^p_k(\Sigma; L^p(D_R))$$ and by lemma $\ref{lem19}$ we get $\alpha \in L^p_{k+1}(\Sigma; D_R)$.  Finally, the uniform bounds on $\psi$, $\phi$ as well as $A$ on $D_{R'}\times \Sigma$ yield inform bounds on $\alpha$ on $D_{R'}\times \Sigma$.  
\end{proof}

\subsubsection{$(t,s)$-Estimates on $u$, $\alpha$}
Consider now the equations 
\begin{equation}
\delbar u=0
\end{equation}
\begin{equation}
\partial_t \alpha +*_2\partial_s \alpha=d_\alpha \phi+*_2d_\alpha \psi 
\end{equation}
It will be convenient at this point to treat the pair $(u,\alpha)$ as a map from $D_R$.  For this, define $$v:D_R\rightarrow \Mod$$ as $v(s,t)=u(s,-t)$.  We may now view the pair $(v,\alpha)$ as a map $$D_R\rightarrow \Mod^- \times \Aff$$ with the Lagrangian boundary condition $(v(0,s),\alpha(0,s))\in \Lag$.  The estimates below will follow by applying our regularity results for Banach valued holomorphic curves.  

\begin{lemma}
\label{lemreg4}
Assume $(u,A) \in L^{p}_k$ with $p>2$ and $k>1$.  We have $(u,A)\in L^p_{k+1}$.   \\
Assume $(u,A) \in L^{p}_1$ with $p>4$.  We have $(u,A)\in L^{p/2}_{2}$. \\
Assume $(u,A) \in L^{p}_1$ with $p>2$.  We have $(u,A)\in L^{1/p_1}_{1}$.    
\end{lemma}

\begin{proof}
To illustrate the proof let us prove the second claim.  The proofs of the other parts are similar.  By lemma $\ref{lemreg1}$ and $\ref{lemreg2}$, we have $\psi, \phi \in L_{1}^{1/p_1}$.  We have $$d_\alpha \phi+*_2d_\alpha \psi\in L^{1/p_1}$$ Applying the change of coordinates from section $\ref{Canlag}$  gives as a map $$\beta=(v',\alpha'):D_R \rightarrow B_p$$ such that $$\partial_t \beta +J(\beta) \partial_s \beta=\gamma $$ where $\gamma $ is $ d_\alpha \phi+*_2d_\alpha \psi$ in the new coordinates.  Since the change of coordinates preserves the $L^{1/p_1}$-structure,  we may use the elliptic regularity lemma  $\ref{Jreg2}$ to deduce that $\beta \in L_{2}^{1/q_1}$.  Applying the change of coordinates in the other direction, we obtain $(u,\alpha) \in L^{1/q_1}_2(D_R; L^{1/q_1}(\Sigma))$ as desired.  
\end{proof}

\subsubsection{Synthesis of Regularity Arguments}

We are now in position to put together the estimates from the previous parts to obtain regularity for the matched equations.  Assume we have a matched pair $(u,A)$ on $D_R\times \Sigma$.  We will obtain regularity and bounds on various Sobolev norms in a small neighborhood of the center of $D_R$.  Note that the size of the neighborhood will depend on the specific norm in question.  In fact, at each elliptic estimate we need to shrink the size of $R$. Since we are interested in a regularity/compactness statement on a given compact set, it suffices to prove that each point $p\in D_R$ has a neighborhood where any given $L^p_k$-norm is bounded.  \\\\
\textbf{Step 1}: Assume $(u,A) \in L^p_1=L^{1/p_0}_1$ where $p_0=1/p<1/2$.  We claim that $(u,A)\in L_{1}^{1/(2p_0-1/2)}$ as long as  $p_0>1/4$.  Let $$p_1=2p_0-1/2$$ and $$q_1=2p_0-1/4$$  First, we apply lemma $\ref{lemreg1}$ to deduce that $\psi\in L^{1/q_1}_2$ and $\psi \in L^{1/p_1}_1$. Now, we use lemma $\ref{lemreg2}$ to deduce that $\phi \in  L^{1/q_1}_2$ and $\phi \in L^{1/p_1}_1$.  Next, we use lemma $\ref{lemreg3}$ to deduce that $\alpha \in L^{1/p_1}_1(\Sigma; L^{1/p_1}(D_R))$.  Finally, we use lemma $\ref{lemreg4}$ to deduce that $\alpha \in   L^{1/p_1}_1(D_R; L^{1/p_1}(\Sigma))$ and $u\in L^{1/p_1}_1$. We obtain, using section $\ref{SobBan}$ that $\alpha \in L_1^{1/p_1}(D_R\times \Sigma)$ as desired.  Let $$\delta_1=p_0-p_1=1/2-p_0>0$$  If we set $$p_k=2p_{k-1}-1/2$$ we obtain that $$\delta_k=1/2-p_{k-1}> \delta_1$$  Thus, after finitely many steps, $1/8< p_k <1/4$.  \\\\      
\textbf{Step 2}: Assume $(u,A) \in L^{1/p_0}_1$ where $1/8<p_0=1/p<1/4$.  Let $q_1=2p_0-1/4$. We claim that $(u,A)\in L_{2}^{p/2}$. First, we apply lemma $\ref{lemreg1}$ to deduce that $\psi\in L^{1/q_1}_2$. Now, we use lemma $\ref{lemreg2}$ to deduce that $\phi \in  L^{1/q_1}_2$.  Next, we use lemma $\ref{lemreg3}$ to deduce that $\alpha \in L^{1/q_1}_2(\Sigma; L^{1/q_1}(D_R))$.  Finally, we use lemma $\ref{lemreg4}$ to deduce that $\alpha \in   L^{p/2}_2(D_R; L^{p/2}(\Sigma))$ and $u\in L_2^{p/2}$.  We have $(u, A)\in L^{p/2}_2$ as desired.  \\\\
\textbf{Step 3}: Assume $(u,A) \in L^p_k$ where $pk>4$, $p>2$ and $k>1$.  We claim that $(u,A)\in L^p_{k+1}$.   First, we apply lemma $\ref{lemreg1}$ to deduce that $\psi\in L^{p}_{k+1}$. Now, we use lemma $\ref{lemreg2}$ to deduce that $\phi \in  L^p_{k+1}$.  Next, we use lemma $\ref{lemreg3}$ to deduce that $\alpha \in L^{p}_{k+1}(\Sigma; L^{p}(D_R))$.  Finally, we use lemma $\ref{lemreg4}$ to deduce that $\alpha \in   L^{p}_{k+1}(D_R; L^p(\Sigma))$ and $u\in L^p_{k+1}$. We have $(u, A)\in L^p_{k+1}$ as desired.  \\\\
 This completes the proof of theorem $\ref{thmreg}$. \\\\
\textbf{Remark.} Given a sequence $(u_i, A_i)$ converging uniformly in $L^p_1$ we may conclude that, for some subsequence, $(v(0,0),\alpha(0,0))$ converges strongly in $L^p$ to a limit $x\in \Mod \times \Aff$.  We may therefore choose a universal chart for all sufficiently large $i$ where we can straighten the Lagrangian boundary conditions using section $\ref{Canlag}$.  Finally, we use elliptic estimates from section $\ref{JRegSec}$ to obtain uniform bounds on the almost complex structure for all $i$ sufficiently large.  This provides a sequential version of theorem $\ref{thmreg}$.

\section{Compactness}
\subsection{Review of Weak Compactness}
\label{weakcom}
Our proof compactness will use the fundamental results of Uhlenbeck (see \cite{Uhlenbeck} as well as the refinements in \cite{Kat1}).  Let $X$ be an oriented, compact, Riemannian 4-manifold possibly with boundary.  Let $A_i$ be a $L^p_1$ sequence of connections on some principal bundle (or associated vector bundle) with a compact structure group $G$.
\begin{theorem}
Take $p>2$, and assume that $||F_{A_i}||_{L^p}$ is bounded.  Then, there exists an $L^p_1$ connection $A$, a subsequence ${A_j}\subset {A_i}$ and $L^p_2$ gauge transformations $g_j$ with the following properties:\\\\
1.  $g_i(A_i)$ are in some fixed $L^p_1(X)$ neighborhood of $A$ and converge $L^p_1$-weakly to $A$ \\
2.  $g_i(A_i)$ converge strongly to $A$ in $L^q$ for  any $4<q<\frac{4p}{4-p}$  \\
\end{theorem}
The compactness theorem is useful in conjunction with the following gauge fixing result: 
\begin{theorem}
Let $A_i$ be a sequence of $L^p_1(X)$ connections converging to $A$ in the weak $L^p_1$ topology.  Then, there exist $L^p_2(X)$ gauge transformations $g_i$, such that $g_i(A_i)$ are in Coloumb-Neumann gauge with respect to $A$.  In other words,
$$d_A^*(g_i(A_i)-A)=0$$
$$*(g_i(A_i)-A)_{|\partial X}=0$$
 \end{theorem}
   
\subsection{Interior Compactness}
Let us briefly recall the compactness results of Gromov and Uhlenbeck.  These are well known and discussed in detail in many texts (see for instance \cite{survey book} and \cite{DK}).  It is worth mentioning that the a priori estimates of this work give an independent proof of these compactness results.  \\\\
Let us first discuss the case of Uhlenbeck compactness.  Fix $X$ as above, and let ${X}^\circ=X-\partial X$.  Let $A_i$ be a sequence of ASD connections on ${X}^\circ$.  By the regularity results, we may assume that, after a gauge transformation, the $A_i$ are smooth.  Suppose that $A_i$ have uniformly bounded energy.  

\begin{theorem}
  There exists a finite sequence of points $p_k\in {X}^\circ$ and a subsequence $A_j$ with the following properties:\\\\
  1.  $A_j$ converge (in any $C^k$-norm) on compact  subsets of ${X}^\circ -\cup_kp_k$ to an ASD connection $A_\infty$. \\\\
2. If $\cup_k p_k$ is nonempty, there exists $\hbar>0$, independent of $A_i$ such that $E(A_\infty)<\liminf E(A_j)-\hbar$ 
\end{theorem}

Now, we turn the Gromov compactness.  Let $D$ be any compact Riemann surface (possibly with boundary).  Let $(M,\omega,J)$ be a compact symplectic manifold with compatible almost complex structure $J$.  Consider a sequence of $J$-holomorphic maps $$u_i:{D}^\circ \rightarrow M$$ with uniformly bounded energy.  As in the ASD case, such maps are automatically smooth as soon as they are $L^p_1$ for $p>2$.   Here is the version of Gromov compactness we need:
\begin{theorem}

There exists a finite sequence of points $p_k\in {D}^\circ$ and a subsequence $u_j$ with the following properties:\\\\
 1. $u_j$ converge (in any $C^k$-norm) on compact  subsets of ${D}^\circ -\cup_kp_k$ to a holomorphic curve $u_\infty$. \\\\
2. If $\cup_k p_k$ is nonempty, there exists $\hbar>0$, independent of $u_i$ such that $E(u_\infty)<\liminf E(u_j)-\hbar$ 

\end{theorem}

\begin{definition}
\label{SingSet1}
A \textbf{singular set} $S$ on $D_R\times \Sigma$ is a finite collection of points $x_i\in (D_R^--\partial D_R^-)$, $y_i\in (D_R-\partial D_R)\times \Sigma$, $z_i\times \Sigma \in (I_R-\partial I_R)\times \Sigma$.  The $z_i\times \Sigma$ are the \textbf{boundary slices} of $S$.      

\end{definition}

\begin{theorem}
Assume that we have a uniform bound $E(u_i,A_i) <C$.  There exists a subsequence $(u_j,A_j)$ and a singular set $S$ with the following properties.  Let $K_0$ be a compact set in $\mathring{D}^-_R-S$ and $K_1$ be a compact set in $\mathring{D}_R \times \Sigma -S$.    We have that $u_j$ converges in any $C^k$ norm on $K_0$ and $A_j$ converges in any $C^k$-norm on $K_1$.   Finally, the energy loss at each singular point is at least  $\hbar$ for some sufficiently small $\hbar>0$ independent of the choice of sequence.  
\end{theorem}   

\begin{proof}
Our first task is to argue that outside some finite singular set we have  $E(D_{R_1}) \leq \hbar$ for all sufficiently small $R_1<R$.  Let us call $p\in D_R $ singular if for any $D_{r}$ with center $p$ we have $$\liminf E_{D_r}(u_i,A_i)\geq \hbar$$  Suppose $p_1$ is such a point.  Pass to a subsequence $(u_j,A_j)$ where $$\lim E_{D_r(p_1)}(u_j,A_j)\geq \hbar$$ for any $r>0$.  Now, consider a different singular point for $(u_j,A_j)$.  Let us call it $p_2$.  We may pass to a subsequence $(u_k,A_k)$ such that $\lim E_{D_r(p_2)}(u_j,A_j)\geq \hbar$ for any $r>0$.  Repeating this $N$ times yeilds $N$ singular points as well as a subsequence which has at least $\hbar$ energy near each singular point.  Since $E(u_i,A_i)$ is bounded, there can be at most a finite number of such singular points.  Thus, we may restrict to proving compactness away from these singular points.  We therefore, consider a disk $D_R\times \Sigma$ where $E(u_i,A_i)\leq \hbar$.  Now, we may apply the a priori estimates of section $\ref{apriori}$ to obtain an $L^\infty$ bound on $du_i$ as well as bounds on $||F_{A_{i}}||_{L^4}$,  $||\nabla_{A_i}F_{A_{i}}||_{L^2}$.  The results of Uhlenbeck in section \ref{weakcom}  imply that we can  put $A_i$ in a Coloumb-Neumann gauge with respect to some smooth connection $A_\infty$. The lemma above implies convergence of $A_i$ in any $L^p_1$ with $p<4$.  Since in dimension 2 the map $L^p_1\rightarrow C^0$ is compact for any $p>2$, we obtain a $C^0$ convergent subsequence for $u_i$. We are now in position to apply the regularity and convergence results of section \ref{regsec} to conclude that we have uniform $C^k$-bounds on $(u_i, A_i)$ in a neighborhood of each point.  This implies that after passing to a subsequence we have $C^{k-1}$-convergence of $(u_i, A_i)$ on any compact set.    
\end{proof}

\subsection{Gromov-Uhlenbeck Compactness} 
We can now state and prove our generalization of the results above.  We will consider a sequence of matched pairs $(u_i,A_i)$ on $D_R \times \Sigma$.  By our regularity results, we may assume that the sequence consists of smooth elements. 
The following convergence result is useful in our discussion of Gromov-Uhlenbeck compactness.  
\begin{lemma}
Suppose $A_i$ are in Coloumb-Neumann gauge with respect to some fixed smooth $A_0$ on $D_R\times \Sigma$.  Furthermore, assume that we have a uniform bound on $||\nabla_{A_i}F_{A_i}||_{L^2}$ and $||A_i-A_0||_{L^4_1}$.  We have that $A_i$ has a strongly $L^p_1$-convergent subsequence for any $2\leq p <4$ on $D_{R/2}\times \Sigma$.    
\end{lemma}
\begin{proof}
By Kato's inequality (see equation ($\ref{Kato}$)), the bound on  $||\nabla_{A_i}F_{A_i}||_{L^2}$ gives a uniform bound on $||F_{A_i}||_{L^4}$.  The embedding $$L^4_1\cdot L^4_1 \rightarrow L^2_1$$ implies that the $L^4_1$-bound on $A_i$ gives us an $L^2_1$-bound on $A_i \wedge A_i$.  Now, $$\nabla_0 F_{A_i}= \nabla_{A_i} F_{A_i}+(A_0-A_i)\cdot F_{A_i}$$ In view of the $L^4$-bound on $F_{A_i}$, we obtain an $L^2$-bound on $\nabla_0 F_{A_i}$.  Since $$d_{0} (A_i-A_0)= F_{A_i}-(A_i-A_0)\wedge(A_i-A_0)-F_{A_0}$$ we obtain a uniform bound on $||d_0 A_i||_{L^2_1}$.  The embedding $$L^2_1\rightarrow L^p$$ is compact for all $p<4$, therefore $d_0 A_i$ is strongly precompact in $L^p$.  Now, since $A_i$ are in Coloumb-Neumann gauge, we use the estimate $$||A_i-A_j||_{L^p_1}\leq C(||A_i-A_j||_{L^p}+||d^*_0 (A_i-A_j)||_{L^p}+||d_0(A_i-A_j)||_{L^p})$$ on $D_{R/2}\times \Sigma$ to deduce that $A_i$ is strongly precompact in $L^p_1(D_{R/2}\times \Sigma)$.   

\end{proof}

\section{Removal Of Singularities}

\subsection{Statement of results}

In previous sections we have discussed a compactness theorem in the context of matched pairs  $(u,A)$.  As demonstrated, a sequence of pairs with a uniform energy bound converges outside a set of singularities.  Thus, such a sequence gives rise to a matched pair $(u_\infty, A_\infty)$ that has finite energy but is not defined on the entire domain.   For interior singular points $x_i$ and $y_i$ of definition $\ref{SingSet1}$, we can complete the pair $(u_\infty, A_\infty)$ using removal of singularities for $J$-curves and ASD equations (see \cite{MS} and \cite{DK}).  It remains to address the singularities at the boundary slices  $z_i$.  \\\\
Let $D_R$ be the disk $$D_R=\{ (s,t)\in \C | s^2+t^2 \leq R^2, s \geq 0\}$$ and let $D_R^*=D_R-(0,0)$.
Consider a matched pair $(u,A)$ defined on $D_R^*$ as in section $\ref{MatchB}$.  We assume $\Energy(u,A) <\infty$.  
\begin{theorem}
\label{1223}
There exists a matched pair $(u',A')$ on $D_R$ that is gauge equivalent to $(u,A)$ on $D_R^*$.  The pair $(u',A')$ is said to \textbf{extend} $(u,A)$. 
\end{theorem}
This theorem completes our framework of Gromov-Uhlenbeck compactness.  We see that a sequence of pairs weakly converges to a limiting pair.  In case some $\hbar>0$ of energy concentrates at singular points we get strong (in any norm) convergence outside this set.    The limiting object has strictly smaller energy in case the singular set is nonempty.  The remainder of this section is devoted to a proof of theorem $\ref{1223}$.  The proof is reminiscent of removal of singularities for ASD connections with Lagrangian boundary conditions discussed in \cite{Kat3}.

\subsection{The Chern-Simons Functional}
Given a closed 3-manifold $Y$ with a $U(2)$-bundle $P$ and connection (on the associated SO(3)-bundle) $A$ we may define the Chern-Simons invariant as follows.  Pick a flat base connection $A_0$ and let $B=A-A_0$.  Chern-Simons invariant is defined as 
\begin{equation}
CS(A)=tr(\int_YB\wedge d_{A_0}B +\frac{2}{3}B^3)=tr(\int_YB\wedge F_A -\frac{1}{3}B^3)
\end{equation}
The importance of Chern-Simons for us comes from the following.  Let $Y=\partial X$ and assume $A_0$ extends to a flat connection on $X$ and $A$ extends to a connection on $X$. We have
\begin{equation}
\begin{split}
CS(B)&=tr(\int_X d_{A_0}A\wedge d_{A_0} A+\frac{2}{3}(( d_{A_0}A)A^2-A( d_{A_0}A)A+A^2 d_{A_0}A)) \\
&=tr(\int_X  d_{A_0}A\wedge  d_{A_0} A+2A d_{A_0}A)\\
&=tr(\int_XF_A\wedge F_A)
\end{split}
\end{equation}

 Now, assume that $Y=\Sigma\times S^1$ and take some flat connection $\alpha_0$ on $\Sigma$.  We can pull it back to obtain $A_0$ on $Y$.  If $B=\alpha +\phi ds$  then,
\begin{equation*}
 \begin{split}
 B^2&=\alpha^2+[\alpha,\phi ]ds \\
 B^3&=\alpha[\alpha,\phi ]ds+\phi \alpha^2ds \\
  d_{A_0}B&= d_{\alpha_0}\alpha + d_{\alpha_0}\phi ds -\partial_s \alpha ds \\
  d_{A_0} B\wedge B&=( d_{\alpha_0}\alpha) \phi ds +\partial_s \alpha \wedge \alpha ds- (d_{\alpha_0}\phi) \alpha ds
 \end{split}  
 \end{equation*}
 Since $tr(B^3)=3tr(\phi \alpha^2)ds$  and  $$ d_{\alpha_0}(\phi \alpha )=( d_{\alpha_0}\phi )\alpha ds =\phi  d_{\alpha_0}\alpha ds $$ and we get
\begin{equation}
CS(\alpha+\phi ds)=\int_{S^1}\int_\Sigma 2tr(\phi F_\alpha)+tr(\partial_s \alpha \wedge \alpha) 
\end{equation}  
Thus, if $F_\alpha=0$ or $\phi=0$, we are reduced to the canonical 1-form of symplectic geometry.  \\\\
Consider now a matched pair $(u,A)$.  By taking $R$ small, we may assume $\Energy(u,A)$ is as small as we wish.  We regard $A$ as living on $D_R^+\times \Sigma$ where $D_R^+$ is the positive half disk and $u$ as defined on $D_R^-$.   Take $(r,\phi)$ to be polar coordinates on $\C$.   Given $(r,\phi) \in D^*_{R/2}$ we take $D_{r}(r,\phi)$ centered at $(r,\phi)$ that is completely contained in $D_{R}^*$.   If the energy for $D_R$ is small, we apply theorem $\ref{thm2.1}$ to obtain  $$||F_A(r,\phi)||_{L^2(\Sigma)}\leq Cr^{-1}(\Energy(u,A))^{1/2}$$ and $$|du(r,\phi)|\leq Cr^{-1}(\Energy(u,A))^{1/2}$$  Thus, given $\ep>0$ for $r$ sufficiently small, $||F_A(r,\phi)||_{L^2(\Sigma)}\leq \ep/r $ and $|du(r,\phi)|\leq \ep/r$.   It follows that on $S_r^-$, $u$ is contained in a single chart where it can be written as $\alpha+\alpha_0$ with $\alpha_0$ a fixed flat connection and $\alpha$ some 1-form.  We choose a chart where $\alpha(r,\pi/2)=0$ and $d^*_{\alpha_0}\alpha=0$.   With this choice of lift we have $$|du|\geq C'|\nabla_{\alpha_0}\alpha|$$  We may trivialize $A$ to have the form $\alpha_0 +\alpha^++\beta dr$ with no $d \phi$ component.  By the matching condition, we can assume $\alpha^+(0,\pi/2)=0$.  Thus, the connections on the two sides coincide at that point.  On the other hand, we have $$g^*_r(\alpha_0+\alpha^-(r,-\pi/2))=\alpha_0+\alpha^+(r,-\pi/2)$$ where $g_r$ is a gauge transformation of connections on $\Sigma$.    Energy of the connection is then expressed as $$\int_{D_r^+}||F_A||^2_{\Sigma}+\rho^{-2}||\partial_\phi A||^2_{\Sigma}\rho d \rho d\phi $$  This implies that $||\partial_\phi A||\leq \ep$.  Thus, $$||A(\pi/2,r)-A(-\pi/2,r)||_{L^2} \leq \ep $$ for all $r$ small.  Similarly, the energy of $u$ is exwe have $$||\alpha^-(\pi/2,r)-\alpha^-(-\pi/2,r)||_{L^2} \leq \ep $$

\begin{lemma}
Given $g\in \Gau(\Sigma)$ there exists an extension $\tilde{g}\in \Gau(\Sigma\times [0,1])$ such that $\tilde{g}_{|\{0\}}=g$,  $\tilde{g}_{|\{1\}}=Id$ and $||\tilde{g}||_{L^3_1}\leq C||g||_{L^2_1}$ for some universal $C>0$.  
\end{lemma}
\begin{proof}
This result is based on a theorem of Hang and Lin \cite{HL} and is discussed in detail in \cite{Kat3}.   
\end{proof}
Setting $\tau=\tilde{g}^*(\alpha^-+\alpha^-(\rho,-\pi/2))$ we obtain a flat connection on $[0,1]\times \Sigma$ such that $$||\tau-\alpha^-(\rho,-\pi/2)||_{L^3;[0,1]\times \Sigma}\leq C' ||\alpha^-(\rho,-\pi/2)-\alpha^+(\rho,-\pi/2)||_{L^2;\Sigma} $$ By taking a close approximation (in $L^3_1$) of $\tilde{g}$, we may assume that it is constand near the boundary in the transverse direction.  This is helpful for patching connections together.  \\\\
We built a closed 3-manifold $Y_r$ as follows.  Join $S_r^+\times \Sigma$ with $S_r^-\times \Sigma$ along $(r,\pi/2)\times \Sigma$.  Glue in $Z=[0,1] \times \Sigma$ by identifying $0\times \Sigma$ with $(r,-\pi/2)\times \Sigma$ on $S^-_r$ and  $1\times \Sigma$ with $(r,-\pi/2)\times \Sigma$ on $S^+_r$.  By construction we obtain a connection $A'$ on the 3-manifold $Y_r$.  The Chern-Simons on $Y_r$ with respect to $\alpha_0$ is given by $$\frac{-1}{3}\int_{[0,1]\times \Sigma }tr(\tau-\alpha_0)^3+\int_{-\pi/2}^{\pi/2}\int_{\Sigma} tr(\alpha^+\wedge \partial_\phi\alpha^+)d\phi d\Sigma+\int_{\pi/2}^{3\pi/2}\int_{\Sigma} tr(\alpha^-\wedge \partial_\phi\alpha^-)d\phi d \Sigma$$

We have $$|\int_{[0,1]\times \Sigma }tr(\tau-\alpha_0)^3|\leq (\int_{-\pi/2}^{\pi/2}\int_\Sigma|\partial_\phi \alpha^+|)^3+(\int_{\pi/2}^{3\pi/2}\int_\Sigma|\partial_\phi \alpha^-|)^3$$ since $|du|$ controls any $L^p$-norm of the lift $\alpha$.  
Since $\alpha^\pm(\pi/2,r)=0$ we obtain

$$\int_{\pi/2}^{3\pi/2}\int_{\Sigma} tr(\alpha^-\wedge \partial_\phi\alpha^-)d\phi d \Sigma=\int_{\pi/2}^{3\pi/2}\int_{\Sigma} \int_{\pi/2}^\phi tr(\partial_v\alpha^-\wedge \partial_\phi\alpha^-)dvd\phi d \Sigma$$
$$\int_{-\pi/2}^{\pi/2}\int_{\Sigma} tr(\alpha^+\wedge \partial_\phi\alpha^+)d\phi d\Sigma=-\int_{-\pi/2}^{\pi/2}\int_{\Sigma} \int^\phi_{-\pi/2}tr(\partial_v\alpha^+\wedge \partial_\phi\alpha^+)dvd\phi d\Sigma$$ 

Thus, we can bound $CS(A')$ by $$Cr(\int_{S^+_r}||F_A||^2_{\Sigma}+r^{-2}||\partial_\phi A||^2_{\Sigma}+\int_{S_r^-}|du|^2)=Cr\partial_r\Energy(r)\leq  \ep^2$$

\subsection{Isoperimetric Inequality}
From the previous section we have concluded that for the specific choice of Chern-Simons we have the estimate 
\begin{equation}
\label{isop}
CS(A')\leq  Cr^{}\partial_r\Energy(r)
\end{equation}
 
   To obtain the Isomperimetric Inequality we must now relate $CS(A')$ to $\Energy(\rho)$.
Let $$D_{\delta \rho}^{\pm}=\{(s,t)\in D_{\rho}^+| s^2+t^2 \geq \delta\}$$
Given $0< \delta< \rho$ we define a 4-manifold $X_{\delta \rho}$ by taking the union of all $Y_t$ for $t\in [\delta, \rho]$.  Thus,  $X_{\delta \rho}$ consists of 3 pieces. $D_{\delta \rho}^{\pm}\times \Sigma$ and $[0,1]\times  [\delta, \rho] \times \Sigma$. We define a connection $\tilde{A}$ on $X_{\delta \rho}$ as follows. On $D_{\delta \rho}^{+}\times \Sigma$ take $A$ in the gauge where $A=\alpha^++\beta dr$.  On  $D_{\delta \rho}^{-}\times \Sigma$ we take a lift $\alpha^-$ of $u$ as above such that $\alpha^+(r,\pi/2)=\alpha^-(r,\pi/2)$.  We have $g_r^*\alpha^-=\alpha^+$ at $\phi=-\pi/2$.  We take any smooth extension $\tilde{g_r}$ of $g_r$ to $[0,1]\times [\delta,\rho]\times \Sigma$ as above with the condition that at $r=\rho$ the extension agrees with the one for $Y_r$.  On $[0,1]\times [\delta,\rho]\times \Sigma$ we set the connection to be $\tilde{g_r} \alpha^-$ and extend $\beta$ arbitrarily.
On $X_{\delta, \rho}$ we have $$\int_{X_{\delta, \rho}}F_{\tilde{A}}^2=\Energy(A_{D^+_{\delta, \rho}})+\Energy(u_{D^-_{\delta, \rho}})$$ since $\tilde{A}$ is flat on $[0,1]\times [\delta, \rho]\times \Sigma$.  Relating the energy to Chern-Simons of the boundary, we obtain $$\int_{X_{\delta, \rho}}F_{\tilde{A}}^2=CS(A'_\rho)-CS(A''_\delta)$$ where $A''$ is the restriction of $\tilde{A}$ to $[0,1]\times \delta, \times \Sigma$.   A priori, $CS(A''_{\rho})$ may differ the the definition of the previous section by a multiple of $4\pi^2$.  However, we see that $CS(A''_\delta)=\Energy_{X_{\delta, \rho}}-CS(A'_\rho)$ is arbitrarily small when $R$ is small and thus is specified uniquely.    Thus, by taking the limit as $\delta \rightarrow 0$, $CS(A''_\delta)\rightarrow 0$ and we obtain the desired formula $CS(A'_r)=\Energy(r)$.  This gives the desired inequality $\ref{isop}$ and thus $\Energy(r)\leq C'r^{\beta}$ where $\beta>0$.

\subsection{Completing the Proof}
So far we have deduced an energy decay $\Energy(r)\leq Cr^{2\beta}$ for a matched pair on a punctured disk.  Let us now use this decay to complete theorem $\ref{1223}$.  First, let us focus on the connection $A$:
\begin{lemma}
There exists $C>0$ such that for all $r$ sufficiently small:\\
$a)  \sup_{\phi}||F_A(r,\phi)||_{L^2(\Sigma)}\leq C'r^{\beta-1}$\\
$b)  \sup_{\phi}||F_A(r,\phi)||_{L^\infty(\Sigma)}\leq C'r^{\beta-2}\cos(\phi)^{-2}$
\end{lemma}
\begin{proof}
To show a),  begin by taking $r_0$ small, we may assume that $\Energy(A)\leq \hbar$.  Given $(r,\phi) \in D^*_{r_0/2}$ we take $D_{r}(r,\phi)$ that is completely contained in $D_{2r}^*$.   We now apply theorem $\ref{thm2.1}$ to obtain  $$||F_A(r,\phi)||_{L^2(\Sigma)}\leq Cr^{-1}(\Energy(A_{|D_{2r}^*}))^{1/2}\leq C'r^{\beta-1}$$ as desired.  For b), we take a point $(x,y)$ on $\Sigma$ and fix $(r, \phi)$.  The 4 dimensional ball $B_{r\cos(\phi)/2}$ centered at $(x,y,r,\phi)$ is contained in $D^*_{r}\times \Sigma$ and by 4-dimensional  analysis of the ASD equation on a ball (see \cite{DK})  we obtain $$||F_A(r,\phi)||_{L^\infty(\Sigma)}\leq C r^{-2}\cos(\phi)^{-2}\Energy(A_{|D_{2r}^*}) \leq Cr^{\beta-2} \cos(\phi)^{-2}$$
\end{proof}
We now cite the following result from \cite{Kat3}:
\begin{theorem}
Let $A$ satisfy a) and b) from the previous lemma.  For some $p>2$, there exists a gauge transformation $g\in L^p_1(D^*\times \Sigma)$ such that $g^*A$ extends to an $L^p_1$ connection on $D\times \Sigma$.  
\end{theorem}

By continuity such an extension $A'$ must satisfy the ASD equation on $D\times \Sigma$.  We now turn to extending $u$.  By our energy decay, we have $$|du(r,\phi)|\leq Cr^{\beta-1}$$ as in the lemma above.  This implies that $u$ extends to $u'$ on $D$ as a H\"{o}lder map and we have $du'\in L^p$ for some $p>2$.  This implies that in fact $u'\in L^p_1$ and thus, by continuity, $u$ is $J$-holomorphic on $D$.  Finally $(u,A)$ is a matched pair since the Lagrangian matching condition is a closed condition on such pairs.

\end{document}